\documentclass{amsart}
\usepackage{amsmath, amssymb, amscd, amsthm, amsfonts, amstext, enumitem,
  amsbsy,  mathrsfs, hyperref,  upgreek, mathtools, stmaryrd, mathabx}
\usepackage[all]{xy}
\usepackage[shadow]{todonotes}

\newenvironment{enumerate-(a)}{\begin{enumerate}[label={\upshape (\alph*)}, leftmargin=2pc]}{\end{enumerate}}

\newenvironment{enumerate-(a)-r}{\begin{enumerate}[label={\upshape (\alph*)}, leftmargin=2pc,resume]}{\end{enumerate}}

\newenvironment{enumerate-(A)}{\begin{enumerate}[label={\upshape (\Alph*)}, leftmargin=2pc]}{\end{enumerate}}

\newenvironment{enumerate-(A)-r}{\begin{enumerate}[label={\upshape (\Alph*)}, leftmargin=2pc,resume]}{\end{enumerate}}

\newenvironment{enumerate-(i)}{\begin{enumerate}[label={\upshape (\roman*)}, leftmargin=2pc]}{\end{enumerate}}

\newenvironment{enumerate-(i)-r}{\begin{enumerate}[label={\upshape (\roman*)}, leftmargin=2pc,resume]}{\end{enumerate}}

\newenvironment{enumerate-(I)}{\begin{enumerate}[label={\upshape (\Roman*)}, leftmargin=2pc]}{\end{enumerate}}

\newenvironment{enumerate-(I)-r}{\begin{enumerate}[label={\upshape (\Roman*)}, leftmargin=2pc,resume]}{\end{enumerate}}

\newenvironment{enumerate-(1)}{\begin{enumerate}[label={\upshape (\arabic*)}, leftmargin=2pc]}{\end{enumerate}}

\newenvironment{enumerate-(1)-r}{\begin{enumerate}[label={\upshape (\arabic*)}, leftmargin=2pc,resume]}{\end{enumerate}}

\newenvironment{enumerate-(Ia)}{\begin{enumerate}[label={\upshape (I\alph*)}, leftmargin=2pc]}{\end{enumerate}}

\newenvironment{enumerate-(IIa)}{\begin{enumerate}[label={\upshape (II\alph*)}, leftmargin=2pc]}{\end{enumerate}}

\newenvironment{enumerate-(IIc)}{\begin{enumerate}[label={\upshape (IIc)}, leftmargin=2pc]}{\end{enumerate}}

\newenvironment{enumerate-(1a)}{\begin{enumerate}[label={\upshape (1\alph*)}, leftmargin=2pc]}{\end{enumerate}}

\newenvironment{enumerate-(2a)}{\begin{enumerate}[label={\upshape (2\alph*)}, leftmargin=2pc]}{\end{enumerate}}

\newenvironment{enumerate-(2c)}{\begin{enumerate}[label={\upshape (2c)}, leftmargin=2pc]}{\end{enumerate}}

\newcommand{\F}{\mathcal{F}}
\newcommand{\RR}{\mathbb{R}}

\newcommand{\pre}[2]{{}^{#1} #2}
\newcommand{\seq}[2]{(#1)_{#2}}

\newcommand{\D}{\mathsf{D}}
\newcommand{\Dec}{\mathsf{dec}}

\newcommand{\id}{\operatorname{id}}

\newcommand{\gr}{\operatorname{graph}}
\newcommand{\cl}{\operatorname{cl}}
\newcommand{\dom}{\operatorname{dom}}
\newcommand{\range}{\operatorname{range}}

\newtheorem{theorem}{Theorem}[section]
\newtheorem{lemma}[theorem]{Lemma}
\newtheorem{corollary}[theorem]{Corollary}
\newtheorem{proposition}[theorem]{Proposition}
\newtheorem{conjecture}[theorem]{Conjecture}
\newtheorem{question}[theorem]{Question}

\theoremstyle{definition}

\newtheorem{definition}[theorem]{Definition}
\newtheorem{fact}[theorem]{Fact}
\newtheorem{example}[theorem]{Example}
\theoremstyle{remark}
\newtheorem{remark}[theorem]{Remark}

\begin{document}

\title[Finite level and \(\omega\)-decomposable Borel functions]{On the structure of finite level and \(\omega\)-decomposable Borel functions}
\date{\today}
\author{Luca Motto Ros}
\address{Albert-Ludwigs-Universit\"at Freiburg \\
Mathematisches Institut -- Abteilung f\"ur Mathematische Logik\\
Eckerstra{\ss}e, 1 \\
 D-79104 Freiburg im Breisgau\\
Germany}
\email{luca.motto.ros@math.uni-freiburg.de}
\thanks{The author warmly thanks M.~Sabok for many useful discussions on this topic, and the anonymous referee for many comments, suggestions, and historical remarks which considerably improved the quality of the paper.}

\begin{abstract}
We give a full description of the structure under inclusion of all finite level Borel classes of 
functions, and provide an elementary proof of the well-known 
fact that not every Borel function can be written as a countable union of 
$\boldsymbol{\Sigma}^0_\alpha$-measurable functions (for every fixed $1 
\leq \alpha < \omega_1$). Moreover, we present some results concerning 
those Borel functions which are $\omega$-decomposable into continuous functions (also 
called countably continuous functions in the literature): such results should be viewed as a contribution 
towards the goal of 
generalizing a remarkable theorem of Jayne and Rogers to all finite levels,
and in fact  they allow us to prove some restricted forms of such generalizations. We also analyze finite level Borel 
functions in terms of composition of simpler functions, and we finally present an application to Banach space theory.
\end{abstract}

\maketitle

\section{Introduction}

The study of \(\omega\)-decomposable Borel 
functions, i.e.\ Borel functions that can be written as countable unions of 
strictly simpler functions, was originated by an old question of Lusin
asking whether every Borel function is \(\omega\)-decomposable 
into \emph{continuous} functions%
\footnote{In the literature, functions admitting such a 
decomposition are usually called countably continuous. However, we preferred to 
use the new terminology (which is closer to the one adopted in~\cite{soldecomposing}) because we will also consider other kinds of countable 
decompositions.}. 
Such question was answered negatively in~\cite{keldys,adynov}, 
and until now many other counterexamples have appeared in 
the literature further showing that in fact 
the ``typical'' (in a certain precise technical sense which will not be discussed here) real-valued Baire class \( 1 \) 
function is \emph{not} \(\omega\)-decomposable into continuous functions~\cite{cm, jacmau, milpol, dar}. A 
particularly simple counterexample to Lusin's question, which 
was in principle available already in the Thirties, is 
the following.

\begin{example} \label{exisomorphism}
Let \( f \) be any Borel isomorphism
between the Hilbert cube \( [0,1]^\omega \)  (or any other Polish space of topological dimension \( \infty \)) and the Baire space 
\( \pre{\omega}{\omega} \) --- such an \( f \) can 
be taken so that both \( f \) and \( f^{-1} \) are Baire class 
\( 1 \) functions by~\cite[p.~212]{kur}. Then at least one of \( f , f^{-1} \) is not \(\omega\)-decomposable into continuous functions. In fact, assume towards a contradiction that  \( f = \bigcup_{n \in \omega } f_n \) and 
\( f^{-1} = \bigcup_{n \in \omega } g_n \) with all the \( f_n \)'s and 
\( g_n \)'s continuous partial functions. Setting 
\( X_{n,m} = \dom(f_n) \cap \range(g_m) \) and 
\( Y_{m,n} = \dom(g_m) \cap \range(f_n) \), we get that 
\( [0,1]^\omega = \bigcup_{n, m \in \omega} X_{n,m} \), 
\( \pre{\omega}{\omega} = \bigcup_{n,m \in \omega } Y_{m,n} \), and  
\( f \restriction X_{n,m} \) is an homeomorphism between \( X_{n,m} \) and 
\( Y_{m,n} \) for every \( n,m \in \omega \). Since all the \( Y_{m,n} \) are 
zero-dimensional, this would imply that \( [0,1]^\omega \) can be written as a 
countable union of zero-dimensional spaces, contradicting a classical
result of Hurewicz  on Polish spaces of topological dimension \( \infty \)  (see e.g.\ \cite[pp.\ 50-52]{hurwal}). 
\end{example}

Lusin's question can be naturally generalized by replacing continuous functions with 
\( \boldsymbol{\Sigma}^0_\alpha \)-measurable functions (for some fixed 
\( 1 \leq \alpha < \omega_1 \)), namely:

\begin{question}[Generalized Lusin's question] \label{questionLusin}
Let \( 1 \leq \alpha < \omega_1 \). Is it true that every Borel function between two uncountable Polish spaces is 
\(\omega\)-decomposable into \( \boldsymbol{\Sigma}^0_\alpha \)-measurable 
functions (i.e.\ it can be written as a countable union of \( \boldsymbol{\Sigma}^0_\alpha \)-measurable partial functions)?
\end{question}

\noindent
According 
to ~\cite{cm}, also this question was answered negatively in~\cite{keldys} for functions from 
\( \pre{\omega}{\omega} \) to the unit interval \( I = [0,1] \),  and  in an unpublished preprint of Laczkovich for functions from any 
uncountable Polish space to \( I \). Even stronger (in a technical sense) 
counterexamples were presented in~\cite{cm, cmps} using the universal functions method.

Despite the fact that Question~\ref{questionLusin} has been answered negatively,
for some natural classes of functions there are positive 
results asserting that certain (definable) decompositions into countably many continuous 
functions are always possible. Given \( 1 \leq \beta \leq \alpha < \omega_1 \), a 
function \( f \colon X \to Y \) will be called \emph{\( \Sigma_{\alpha,\beta} \) 
function} if \( f^{-1}(S) \in \boldsymbol{\Sigma}^0_\alpha(X) \) for every 
\( S \in \boldsymbol{\Sigma}^0_\beta(Y) \). The following remarkable 
theorem was first proved by Jayne and Rogers in~\cite[Theorem 5]{jayrog}, but see 
also~\cite{motsem,kacmotsem} for a simpler proof and some minor strengthenings of it.

\begin{theorem}[Jayne-Rogers] \label{theorjaynerogers} 
Let \( X,Y \) be separable metrizable spaces with \( X \) analytic,%
\footnote{The original Jayne-Rogers theorem holds also in the broader context of nonseparable metrizable spaces when \( X \) is assumed to be absolute Souslin-\( \mathscr{F} \)  (i.e.\ the counterpart of an analytic space in the realm of nonseparable spaces). However, for the sake of simplicity, here and in the 
subsequent weak and strong generalization of the Jayne-Rogers theorem 
(Conjectures~\ref{conjweakJR} and~\ref{conjstrongJR}) we will just  consider the already relevant and well-studied  case of separable metrizable spaces (see e.g.\ 
\cite{soldecomposing}).} 
and let 
\( f \colon X \to Y \). Then 
\( f \) is a \( \Sigma_{2,2} \) function if and only if
\( f \) can be written as a countable union of continuous  functions with  \( \boldsymbol{\Delta}^0_2(X) \) 
(equivalently, closed) domains.
\end{theorem}

In~\cite{andslo} it was asked whether such result can be generalized to all finite 
levels, namely whether the following statement is true:

\begin{conjecture}[Weak generalization of the Jayne-Rogers theorem] \label{conjweakJR}
Let \( X, Y \) be separable metrizable spaces with \( X \) analytic, and let \( 1 
< n < \omega \). Then 
\( f \colon X \to Y \)  is a \( \Sigma_{n,n} \) function if and only if 
\( f \) can be written as a countable union of continuous  functions with  
\( \boldsymbol{\Delta}^0_n(X) \) (equivalently, \( \boldsymbol{\Pi}^0_{n-1}(X) \)) domains.
\end{conjecture}

\noindent
Notice that the Jayne-Rogers theorem~\ref{theorjaynerogers} cannot be 
extended to infinite levels because, as already observed, there are Baire class 
\( 1 \) functions (hence \( \Sigma_{\omega,\omega} \) functions) which are 
not even \(\omega\)-decomposable into continuous functions. 
Using 
game-theoretic techniques,  
Semmes proved  in his Ph.D. thesis~\cite{semmesthesis} that Conjecture~\ref{conjweakJR}
is true for the level \( n = 3 \) when \( X = Y = \pre{\omega}{\omega} \). 
More precisely, he proved the following results:

\begin{theorem}[Semmes] \label{theorsemmes}
Let \( f \colon \pre{\omega}{\omega} \to \pre{\omega}{\omega} \).
\begin{enumerate-(a)}
\item \label{theorsemmes-1}
\( f \) is a \( \Sigma_{3,2} \) function if and only if \( f \) can be written as a 
countable union of \( \boldsymbol{\Sigma}^0_2 \)-measurable functions with 
\( \boldsymbol{\Delta}^0_3(\pre{\omega}{\omega}) \) (equivalently, \( \boldsymbol{\Pi}^0_2(\pre{\omega}{\omega}) \)) domains.
\item \label{theorsemmes-2}
\( f \) is a \( \Sigma_{3,3} \) function if and only if \( f \) can be written as a 
countable union of continuous functions with 
\( \boldsymbol{\Delta}^0_3(\pre{\omega}{\omega}) \) (equivalently, \( \boldsymbol{\Pi}^0_2(\pre{\omega}{\omega}) \)) domains.
\end{enumerate-(a)}
\end{theorem}

\noindent
Actually, in~\cite{semmesthesis} Theorem~\ref{theorsemmes}\ref{theorsemmes-2} is derived from  
Theorem~\ref{theorsemmes}\ref{theorsemmes-1}. Notice that  
Theorem~\ref{theorsemmes}\ref{theorsemmes-2} can be generalized to all 
Polish spaces of topological dimension \( \neq \infty \), like the Euclidean 
spaces \( \RR^n \) or the Euclidean cubes \( I^n \)  
(\( 1 \leq n < \omega \)), by~\cite[Corollary 4.25]{motschsel}. Similarly, 
using~\cite[Theorem 4.21]{motschsel} and~\cite[Theorem 7.8]{kec}, it is 
easy to see that also Theorem~\ref{theorsemmes}\ref{theorsemmes-1} can 
be generalized to all functions from a Polish space of topological dimension 
\( \neq \infty \) to any zero-dimensional separable metrizable space.

Theorem~\ref{theorsemmes}\ref{theorsemmes-1} suggests to consider the 
following further strengthening of Theorem~\ref{theorjaynerogers}.

\begin{conjecture}[Strong generalization of the Jayne-Rogers theorem] \label{conjstrongJR}
Let \( X ,Y\) be separable metrizable spaces with \( X \) analytic, and let \( 1 < n \leq m < \omega \). Then 
\( f \colon X \to Y \)  is a \( \Sigma_{m,n} \) function if and only if 
\( f \) can be written as a countable union of 
\( \boldsymbol{\Sigma}^0_{m-n+1} \)-measurable functions with 
\( \boldsymbol{\Delta}^0_m(X) \) (equivalently, \( \boldsymbol{\Pi}^0_{m-1}(X) \)) domains.
\end{conjecture}

\noindent
Note that setting \( n = m \) in Conjecture~\ref{conjstrongJR} we get Conjecture~\ref{conjweakJR}.

As a by-product of his representation of classes of functions as strategies in 
corresponding games, using diagonalization arguments  Semmes also showed  in~\cite{semmesthesis}
that the diagram of inclusions of Figure~\ref{figinclusion} holds, where \( \Sigma_{m,n} \) represents the class of all \( \Sigma_{m,n} \) 
functions from \( \pre{\omega}{\omega} \) to itself, arrows represent proper 
inclusions, and no inclusion relation holds between classes which are not connected by (composition of)
arrows.

\begin{figure}[!htbp] 
\centering
\mbox{
\xymatrix{
& & \Sigma_{3,1}\\
& \Sigma_{2,1}   \ar@{->}[ur] \ar@{->}[dr] & \\
\Sigma_{1,1} \ar@{->}[ur] \ar@{->}[dr] & & \Sigma_{3,2} \ar@{->}[uu] \\
& \Sigma_{2,2} \ar@{->}[dr] \ar@{->}[uu] \ar@{->}[ur] & \\
& & \Sigma_{3,3} \ar@{->}[uu] }
}
\caption{The inclusion diagram for the classes \( \Sigma_{m,n} \) for \( m \leq 3 \).}
\label{figinclusion}
\end{figure}
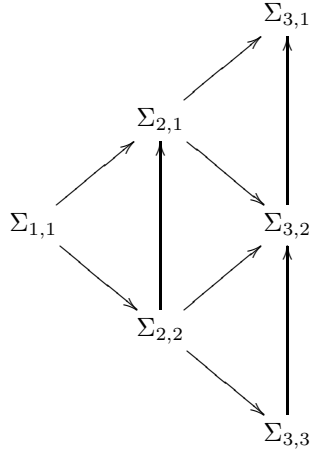

\noindent
It is  then very natural to ask whether such diagram can be extended to 
include at least all the classes \( \Sigma_{m,n} \) for \( 1 \leq n \leq m < \omega \), and 
how such extension looks like.

This paper contains a variety of results which deal with all the above 
mentioned problems, aiming to show that finite level Borel classes of functions
(see Definition~\ref{deffinitelevel})
constitute a rich and interesting subject to be further analyzed.

In Section~\ref{secfinitelevel} we extend  the inclusion 
diagram of Figure~\ref{figinclusion} to all finite level Borel classes.  Although game 
representations of virtually all the classes of functions involved are now 
available by~\cite{motgames}, the present proof does not use such 
representations at all: in fact, instead of using diagonalization arguments over 
the corresponding strategies (as done in~\cite{semmesthesis}), we will 
directly define certain canonical functions which show that all 
obvious inclusions in the resulting diagram are in fact proper, and that no other 
inclusion relation holds 
(Theorem~\ref{theorinclusion} and Figure~\ref{figdiagram}). 

The same ideas will then be used in 
Section~\ref{sectiondecomposable} to provide a general way to construct 
further  counterexamples to the generalized Lusin's question~\ref{questionLusin} 
(Corollary~\ref{cor:luzin}). The technique employed, which differs from the 
universal functions method used in~\cite{cm,cmps}, is rather elementary, and 
 allows to deal not only with real-valued functions%
\footnote{For example, the universal functions method cannot be used for classes of functions ranging into \emph{zero-dimensional} Polish spaces: in fact it can be shown e.g.\ that for every \( 1 \leq \alpha < \omega_1 \) there is no universal function for the class of all \( \boldsymbol{\Sigma}^0_\alpha \)-measurable functions between an uncountable Polish space \( X \) and a zero-dimensional Polish space \( Y \).}
 (as done 
in~\cite{keldys,cm,cmps}), but also with functions between arbitrary 
uncountable analytic spaces. Our counterexamples are constructed in an explicit way
as countable powers of suitable characteristic functions (see Corollary~\ref{cor3} and Corollary~\ref{cor:luzin}): in particular, we get that the obvious generalization of the Pawlikowski 
function (a somewhat canonical example of a Baire class \( 1 \) function which is not \(\omega\)-decomposable into 
continuous functions) for the level \( n \in \omega \) is not \(\omega\)-decomposable into 
\( \boldsymbol{\Sigma}^0_n \)-measurable functions (Corollary~\ref{corPnindecomposable}).

In Section~\ref{sectionJR} we prove slightly varied forms of Conjecture~\ref{conjweakJR} (see~\eqref{eqmain} of Theorem~\ref{theormain} and~\eqref{eqrestrictedJR} of Theorem~\ref{theorJR}), or its restriction to some 
special classes of functions (Corollary~\ref{corsimplecases}). Theorem~\ref{theorJR} also shows that the weak generalization of the Jayne-Rogers theorem is equivalent to a natural condition  concerning the minimal complexity of the graph of functions 
appearing in the relevant Borel classes (Conjecture~\ref{conjweak}). This raises 
natural questions and open problems which are of independent interest, and surely deserve further investigation:  some of them are collected in Question~\ref{quest1}.

In Section~\ref{secstrongJR} we
introduce another 
natural condition (Conjecture~\ref{conjstrong}) which is 
 shown to be equivalent to the strong generalization 
of the Jayne-Rogers theorem (Theorem~\ref{theorsgJR}): such condition is related to certain 
minimal refinements of 
topology in the spirit of~\cite[Theorem 13.1]{kec}.

Section~\ref{seccomposition} characterizes the class \( \Sigma_{3,2} \) in terms of 
composition of simpler functions (namely, of a function in \( \Sigma_{2,1} \) 
and a function in \( \Sigma_{2,2} \)), and shows that if the strong generalization of the Jayne-Rogers theorem (that is, 
Conjecture~\ref{conjstrongJR}) is true then analogous characterizations would hold for all finite level Borel classes. 

Finally, 
Section~\ref{sectionBanach} contains an application of the results from Section~\ref{sectionJR} to Banach space theory, more precisely to the isomorphism problem for Baire classes of functions.

\begin{remark}
We recently learned that a variant of 
Theorem~\ref{theorJR} (together with some instances of its 
Corollary~\ref{corsimplecases}) has been independently obtained using essentially the same 
methods by M.~Sabok and 
J.~Pawlikowski in~\cite[Theorem 1.2 and Corollary 1.3]{pawsabrev}.
\end{remark}

\section{Preliminaries and notation} \label{secpreliminaries}

We will freely use the standard (descriptive) set theoretic notation and terminology --- for all undefined notions and concepts we refer the reader to the standard monograph~\cite{kec}.

\subsection{Ordinals}

Greek letters \( \alpha , \beta , \gamma\) (possibly with various 
decorations) will
always denote \emph{countable} ordinals, while the symbols \( + \) and \( \cdot \) 
will denote, 
respectively, the usual operations of ordinal addition and ordinal multiplication. 
Given 
any \( \alpha, \beta \), we let \( \beta \dotdiv \alpha \) be the unique 
\( \gamma \) 
such that \( \alpha + \gamma = \beta \) if \( \alpha \leq \beta \), and \( 0 \) 
otherwise. Notice that \( \alpha + (\beta \dotdiv \alpha) = \max \{ \alpha, \beta \} \) for all \( \alpha,\beta \).

\subsection{Topological spaces}

A topological space is called \emph{Polish} if it is separable and completely metrizable, and \emph{analytic} if it is (homeomorphic to) an analytic subset of a Polish space. Typical examples of Polish spaces are the \emph{Baire space} \( \pre{\omega}{\omega} \) of all 
\(\omega\)-sequences of natural numbers and the \emph{Cantor space} 
\( \pre{\omega}{2} \) of all binary \(\omega\)-sequences, both endowed with 
the countable product of the discrete topology on \(\omega\) and \( 2 \), respectively. 
To simplify the presentation, given a metrizable space \( Y \) we will denote by \( \widetilde{Y} \) the completion of \( Y \) with respect to some compatible metric on it, 
and regard \( \widetilde{Y} \) as a topological space by endowing it with the metric topology --- the choice of the actual metric used to define \( \widetilde{Y} \) will be irrelevant in all the results below (see e.g.\ Remark~\ref{remmain-2}). Obviously,  \( \widetilde{Y} \) is a Polish space whenever \( Y \) is separable metrizable.

In most of the results presented in this paper we will assume that the spaces under consideration are analytic or Polish: however, it is maybe worth noticing that some of
them (especially the ones asserting the existence of functions with certain special properties) actually hold for 
a wider class of topological spaces --- the interested reader can easily recover from the proofs the exact assumptions needed for each of them.

\subsection{Borel sets and their stratification}

Given a topological space \( (X, \tau	) \) and a set \( A \subseteq X \), we denote by \( \cl(A) \) the closure of \(A \) in \( X\).
 A set \( A \subseteq X \) is called 
\emph{Borel} if it belongs to the minimal \(\sigma\)-algebra containing 
\( \tau	\), the collection of all open sets of \( X \). As explained in e.g.\ \cite[Section 11.B]{kec}, the class of Borel subsets 
of a \emph{metrizable} space \( X \) can be stratified into the cumulative hierarchy consisting of the classes \( \boldsymbol{\Sigma}^0_\alpha(X,\tau) \), \( \boldsymbol{\Pi}^0_\alpha(X,\tau) \), and \( \boldsymbol{\Delta}^0_\alpha(X,\tau) \) (for \( 1 \leq \alpha < \omega_1 \)). Such construction cannot be literally applied to nonmetrizable spaces \( X \), because in this case there could be open subsets of \( X \) which are not \( F_\sigma \). To overcome this difficulty, for arbitrary topological spaces \( (X,\tau) \) we redefine the class \( \boldsymbol{\Sigma}^0_2(X,\tau) \) to be the collection of all sets of the form \(  \bigcup_{n \in \omega} (U^0_n 
\setminus U^1_n ) \) with \( U^i_n \in \boldsymbol{\Sigma}^0_1(X,\tau) \) 
for every \( n \in \omega \) and \( i = 0,1 \), and let the definitions of all other classes in the stratifications be unchanged (note that when \( X \) is metrizable, this definition is equivalent to the classical one by e.g.~\cite[Proposition 22.1]{kec}). It is straightforward to check that with this minor modification one can always recover the usual structure under inclusion of the hierarchy and the classical closure properties of its levels.

Given a topological space \( ( X, \tau ) \) and \( \alpha \geq 1 \), a set 
\(A \subseteq X \) is called \emph{proper 
\( \boldsymbol{\Sigma}^0_\alpha(X) \) set} if 
\( A \in \boldsymbol{\Sigma}^0_\alpha(X, \tau) \setminus 
\boldsymbol{\Pi}^0_\alpha(X,\tau) = \boldsymbol{\Sigma}^0_\alpha(X,\tau) 
\setminus \boldsymbol{\Delta}^0_\alpha(X,\tau) \), and is called 
\emph{ \( \boldsymbol{\Sigma}^0_\alpha(X) \)-hard} if for every zero-dimensional Polish space \( (Y,\tau_Y) \) it holds that for every 
\( B \in \boldsymbol{\Sigma}^0_\alpha(Y,\tau_Y) \) there is a continuous function \( f \colon Y \to X \) such that
\( B = f^{-1}(A) \). A \( \boldsymbol{\Sigma}^0_\alpha(X) \)-hard set 
belonging to \( \boldsymbol{\Sigma}^0_\alpha(X,\tau) \) is called \emph{\( 
\boldsymbol{\Sigma}^0_\alpha(X) \)-complete}; 
it is easy to check that every 
\( \boldsymbol{\Sigma}^0_\alpha(X) \)-complete set is automatically a proper 
\( \boldsymbol{\Sigma}^0_\alpha(X) \)-set. \emph{Proper \( 
\boldsymbol{\Pi}^0_\alpha(X) \) sets}, 
\emph{\( \boldsymbol{\Pi}^0_\alpha(X) \)-hard} sets, and \emph{\( 
\boldsymbol{\Pi}^0_\alpha(X) \)-complete} sets are defined in a similar 
way. In the sequel we will often use the following standard fact.

\begin{fact} \label{factWadge} 
Let \( X \) be a Polish space and \( \alpha \geq 1 \). If \( A \) is a Borel set 
not in \( \boldsymbol{\Sigma}^0_\alpha(X) \) then 
\( A \) is \( 
\boldsymbol{\Pi}^0_\alpha(X) \)-hard (the case \( \alpha = 1 \) is trivial, the case \( \alpha = 2 \) follows from~\cite[Theorem 21.22]{kec}, while the case \( \alpha \geq 3 \) follows from~\cite[Theorem 28.19]{kec}).
In particular, every proper \( 
\boldsymbol{\Sigma}^0_\alpha(X) \) (respectively, proper \( 
\boldsymbol{\Pi}^0_\alpha(X) \)) set is also \( 
\boldsymbol{\Sigma}^0_\alpha(X) \)-complete (respectively, \( 
\boldsymbol{\Pi}^0_\alpha(X) \)-complete), 
\end{fact}

A \emph{(countable) partition of \( X \) in \( 
\boldsymbol{\Delta}^0_\alpha(X, \tau) \)-pieces} is a sequence \( \seq{X_n}{n \in \omega} \) such that \( \bigcup_{n \in \omega} X_n = X \), 
\( X_n \cap X_m = \emptyset \) whenever \( n \neq m \), and \( X_n \in 
\boldsymbol{\Delta}^0_\alpha(X, \tau) \) for every \( n \in \omega \) (and similarly with \( \boldsymbol{\Delta}^0_\alpha(X,
\tau) \) replaced by \( \boldsymbol{\Sigma}^0_\alpha(X,\tau) \) or \( \boldsymbol{\Pi}^0_\alpha(X,\tau) \)). Note that every 
countable partition of \( X \) in \( \boldsymbol{\Sigma}^0_\alpha(X,\tau) \)-pieces is automatically a countable partition of \( 
X \) in \( \boldsymbol{\Delta}^0_\alpha(X,\tau) \)-pieces. Moreover, for  \( \alpha \geq 2 \) every countable partition of 
\( X \) in \( \boldsymbol{\Delta}^0_{\alpha+1}(X,\tau) \)-pieces can be refined to a countable partition of \( X \) in \( 
\boldsymbol{\Pi}^0_\alpha(X,\tau) \)-pieces (the same is true also for \( \alpha = 1 \) if \( (X,\tau) \) is second-countable 
and zero-dimensional). Given \( \alpha \geq 1 \) and two topological spaces \( (X,\tau_X) \), \( (Y,\tau_Y) \), a function \( f \colon X \to Y \) is called \emph{\( \boldsymbol{\Sigma}^0_\alpha \)-measurable} if \( f^{-1}(U) \in \boldsymbol{\Sigma}^0_\alpha(X,\tau_X) \) for every (\( \tau_Y \)-)open \( U \subseteq Y \) (and similarly with \( \boldsymbol{\Sigma}^0_\alpha \) replaced by \( \boldsymbol{\Pi}^0_\alpha \) or \( \boldsymbol{\Delta}^0_\alpha \)).
When 
the space \( X \) or its topology \( \tau \) will be clear from the context, 
we shall omit any reference to them in all the notation and terminology 
introduced above.

\subsection{Countable powers and the Pawlikowski function}

Throughout the paper, every finite or infinite product of topological spaces (hence, in particular, every countable power \( \pre{\omega}{X} \) of a topological space \( X \)) will be always tacitly endowed with the product topology. Moreover, given  \( x \in X \) we denote by \( \vec{x} \) the \(\omega\)-sequence constantly equal to \( x \), which is of course an element of \( \pre{\omega}{X} \).

\begin{definition} \label{def:ctblpower}
Let \(X,Y \) be topological spaces. The \emph{countable power} \( f_\omega \) of the function \( f \colon X \to Y \) is defined by 
\[ f_\omega \colon \pre{\omega}{X} \to \pre{\omega}{Y} \colon \seq{x_m}{m \in \omega} \mapsto
\seq{ f(x_m) }{ m \in \omega }.
\]
\end{definition}

It is not hard to check that \( f \) is continuous if and only if \( f_\omega \)
is continuous; more generally, for every \( \alpha \geq 1 \) we get that \( f \) is \( \boldsymbol{\Sigma}^0_\alpha \)-measurable if and only if \( f_\omega \) is \( \boldsymbol{\Sigma}^0_\alpha \)-measurable%
\footnote{The nontrivial direction of the equivalence may be proved by an argument as in Lemma~\ref{lemma3}.} 
(see Section~\ref{secfinitelevel} for the relevant definitions). Fix a bijection $\langle \cdot,\cdot \rangle\colon \omega \times \omega \to \omega$.  For each \( n \in \omega \), define the ``projection'' 
\[  \pi^{X}_n \colon \pre{\omega}{X} \to \pre{\omega}{X} \colon \seq{ x_m }{ m \in \omega } \mapsto \seq{ x_{\langle n,m \rangle} }{ m \in \omega }. \]
Clearly, every projection
 is
surjective, continuous and open. Moreover, every countable power \( f_\omega \) commutes with the projections, i.e.\ \( \pi^Y_n \circ f_\omega = f_\omega \circ \pi^X_n \) for every \( f \colon X \to Y \) and every \( n \in \omega \).

The Pawlikowski function \( P  \) is the countable power \( P =  f_\omega \colon 
\pre{\omega}{(\omega+1)} \to \pre{\omega}{\omega} \) of the function \( f 
\colon \omega+1 \to \omega \) defined by \( f(n) = n+1 \) if \( n \neq 
\omega \) and \( f(\omega) = 0 \), where \( \omega+1 \) is endowed with the 
\emph{order} topology and \(\omega\) with the discrete topology. 
By~\cite[Theorem 4.1]{soldecomposing}  (see 
Theorem~\ref{theorsoleckidichotomy} below), \( P \) is in a sense the 
canonical example of a Baire class \( 1 \) function which is not 
\(\omega\)-decomposable into continuous functions.

\begin{definition}
Let \( X, Y, X', Y' \) be topological spaces and \( f \colon X \to Y \), \( g \colon X' \to Y' \). We say that \( f \) is \emph{contained} in \( g \) (\( f \sqsubseteq g \) in symbols) if there are topological embeddings \( \varphi \colon X \to X' \) and \( \psi \colon f(X) \to Y' \) such that \( g \circ \varphi = \psi \circ f \).
\end{definition}

\begin{theorem}[Solecki] \label{theorsoleckidichotomy}
Let \( X, Y \) be separable metrizable spaces with \( X \) analytic. If \( f \colon X \to Y \) is \( \boldsymbol{\Sigma}^0_2 \)-measurable, then either \( f \) is \(\omega\)-decomposable into continuous functions, or else \( P \sqsubseteq f \).
\end{theorem}

The Pawlikowski function \( P \) can also be viewed as the countable power of a sort of characteristic function of the proper \( \boldsymbol{\Pi}^0_1 \) subset \( \{ \omega \} \) of \( \omega +1 \). This suggests the following definition.

\begin{definition} \label{defPn}
For each \( 1 \leq n < \omega \), fix a proper \( \boldsymbol{\Pi}^0_n(\pre{\omega}{2}) \) set \( C_n \), and consider its characteristic function \( \chi_{C_n} \colon \pre{\omega}{2} \to 2 \) mapping \( x \in \pre{\omega}{2} \) to \( 1 \) if and only if \( x \in C_n \). 
The \emph{generalized Pawlikowski function} \( P_n \colon \pre{\omega}{(\pre{\omega}{2})} \to \pre{\omega}{2} \) is then defined as
\[
P_n = (\chi_{C_n})_\omega.
 \]
\end{definition}

In particular, since \( P_n \) is the countable power of the characteristic function of a 
proper \( \boldsymbol{\Pi}^0_n(\pre{\omega}{2}) \) set, it is a 
\( \boldsymbol{\Sigma}^0_{n+1} \)-measurable 
(equivalently, a Baire class \( n \)) function (see Proposition~\ref{proppower}). We will prove in Corollary~\ref{corPnindecomposable} that the 
functions  \( P_n \) correctly generalize the Pawlikowski function \( P \), in that they cannot be written as countable 
unions of \( \boldsymbol{\Sigma}^0_n \)-measurable functions. It is maybe worth noting that when \( n \geq 3 \) the functions \( P_n \) 
are in a sense canonical, i.e.\ that they do not depend in an essential way on the choice of the set \( C_n \) used
to define them, as shown by the next proposition. Following~\cite[Definition 8]{jayrog}, call a bijection \( f \) between two metrizable spaces \( X,Y \) 
\emph{first level Borel isomorphism} if both \( f \) and \( f^{-1} \) map \( F_\sigma \) sets to \( F_\sigma \) sets. This is equivalent to requiring that both \( f \) and \( f^{-1} \) are \( 
\boldsymbol{\Delta}^0_2 \)-measurable or, in the terminology of Definition~\ref{defSigma}, that they belong to \( \D_2(X,Y) \) and \( 
\D_2(Y,X) \), respectively. Notice that if \( \alpha \geq 2 \) then the classes \( 
\boldsymbol{\Sigma}^0_\alpha, \boldsymbol{\Pi}^0_\alpha, \boldsymbol{\Delta}^0_\alpha \) are all preserved by first level Borel isomorphisms.

\begin{proposition} \label{propPn}
Let \( 3 \leq n < \omega \) and \( C_n,C'_n \) be two proper \( \boldsymbol{\Pi}^0_n(\pre{\omega}{2}) \) sets. Set \( P_n = (\chi_{C_n})_\omega \) and \( P'_n = (\chi_{C'_n})_\omega \).
\begin{enumerate-(a)}
\item \label{propPn-1}
There are topological embeddings \( g_0, g_1 \colon \pre{\omega}{(\pre{\omega}{2})} \to \pre{\omega}{(\pre{\omega}{2})} \) such that \( P_n = P'_n \circ g_0 \) and \( P'_n = P_n \circ g_1 \). In particular, \( P_n \sqsubseteq P'_n \) and \( P'_n \sqsubseteq P_n \).
\item \label{propPn-2}
There exists a first level Borel isomorphism
\( g \colon \pre{\omega}{(\pre{\omega}{2})} \to \pre{\omega}{(\pre{\omega}{2})} \) such that
\[ 
P_n = P'_n \circ f. 
 \] 
\end{enumerate-(a)}
\end{proposition}

\begin{proof}
To prove~\ref{propPn-1} it is clearly enough to construct the desired \( g_0 \), as the embedding \( g_1 \) can be obtained in the same way by switching the role of \( C_n \) and \( C'_n \).
Since \( n \geq 3 \) and  \( C'_n \), being a proper \( \boldsymbol{\Pi}^0_n(\pre{\omega}{2}) \) set, is \( \boldsymbol{\Pi}^0_n(\pre{\omega}{2}) \)-hard (see Fact~\ref{factWadge}), by e.g.~\cite[Exercise 26.11]{kec} we get that there is a topological embedding \( f \colon \pre{\omega}{2} \to \pre{\omega}{2} \) such that \( f^{-1}(C'_n) = C_n \). Then \( g_0 = f_\omega \colon \pre{\omega}{(\pre{\omega}{2})} \to \pre{\omega}{(\pre{\omega}{2})} \) is a topological embedding such that \( P_n = P'_n \circ g_0 \).

Part~\ref{propPn-2} is proved using a Schr\"oder-Bernstein argument as in~\cite[Lemma 4.8]{motschsel}. Let \( g_0, g_1 
\) be as in~\ref{propPn-1}. Recursively define the sets \( X_n, Y_n \subseteq \pre{\omega}{(\pre{\omega}{2})} \) by 
setting \( X_0 = Y_0 = \pre{\omega}{(\pre{\omega}{2})} \), \( X_{n+1} = g_1(Y_n) \) and \( Y_{n+1} = g_0(X_n) \). Set 
also \( X_\omega = \bigcap_{n \in \omega} X_n \), \( Y_\omega = \bigcap_{n \in \omega} Y_n \), \( X' = X_\omega \cup \bigcup_{i \in \omega} (X_{2i} \setminus X_{2i+1}) \), and  \( Y' = Y_\omega 
\cup \bigcup_{i \in \omega} (Y_{2i+1} \setminus Y_{2i+2}) \). It is easy to inductively 
check that since \( \pre{\omega}{(\pre{\omega}{2})} \) is a compact (Hausdorff) space and both \( g_0 \) and \( g_1 \) 
are continuous, then all of \( X_n,Y_n,X_\omega,Y_\omega \) are closed sets, whence \( X', Y' \in \boldsymbol{\Delta}^0_2(\pre{\omega}{(\pre{\omega}{2})}) \). Moreover, \( g_0 \restriction X' \) is an 
homeomorphism between \( X' \) and \( Y' \) and, similarly, \( g_1^{-1} \restriction (X \setminus X') \) is an 
homeomorphism between \( X \setminus X' \) and \( Y \setminus Y' \). It easily follows that the function \( f = (g_0 \restriction X') \cup (g_1^{-1} 
\restriction (X \setminus X')) \) is a first level Borel isomorphism (see also 
Lemma~\ref{lemmadefdecomposable}). Finally, \( P_n = P'_n \circ f \) because  \( P_n = P'_n \circ g_0 \) and \( P'_n = P_n 
\circ g_1 \).
\end{proof}

\begin{remark}
Notice that Proposition~\ref{propPn} cannot be extended to the case \( n = 1 \) (while whether it can be extended to the case \( n = 2 \) is still an open problem). To see this, let \( C'  = \{ x \} \) for some 
\( x \in \pre{\omega}{2} \) and \( C \) be a proper \( \boldsymbol{\Pi}^0_2(\pre{\omega}{2}) \) set containing at least two 
points. Set \( P_1 = (\chi_C)_\omega \) and \( P'_1 = (\chi_{C'})_\omega \): we claim that there is no injective function \( f 
\colon \pre{\omega}{(\pre{\omega}{2})} \to \pre{\omega}{(\pre{\omega}{2})} \) such that \( P_1 = P'_1 \circ f \). To 
see this consider the sets \( C_\omega = \{ \seq{ x_n }{ n \in \omega } \in \pre{\omega}{(\pre{\omega}{2})} 
\mid \forall n \, (x_n \in C ) \} \) and \( C'_\omega = \{ \seq{ x_n }{ n \in \omega } \in \pre{\omega}{(\pre{\omega}
{2})} \mid \forall n \, (x_n \in C' ) \} \). Then \( C_\omega \) has the cardinality of the continuum, while \( C'_\omega \) contains 
exactly one element, namely \( \vec{x} \). Let \( f  \colon \pre{\omega}{(\pre{\omega}{2})} \to \pre{\omega}
{(\pre{\omega}{2})} \) be such that \( P_n = P'_n \circ f \). Then
\[ 
x \in C_\omega \iff P_1(x) = \vec{1} \iff P'_1(f(x)) = \vec{1} \iff f(x) \in C'_\omega,
 \] 
which implies that \( f \restriction C_\omega \) is constantly equal to \( \vec{x} \). In particular, \( f \) cannot be injective.
\end{remark}

\subsection{Changes of topologies}

Part of the results of this paper are heavily based on the possibility of enriching the topology of a space without loosing its main properties. In particular, the following minor variation of~\cite[Theorem 22.18]{kec}, essentially due to Kuratowski, will be used repeatedly in the paper.

\begin{lemma} \label{lemmachangetopology}
Let \( (X, \tau) \) be a metrizable space, \( \alpha > 1 \), and \( \mathcal{A} = \{ A_n \mid n \in \omega \} \subseteq \boldsymbol{\Sigma}^0_\alpha(X,\tau) \). Then \(\tau\) can be refined by a topology \( \tau' \subseteq \boldsymbol{\Sigma}^0_\alpha(X,\tau) \) such that \( \mathcal{A} \subseteq \tau' \) and \( Z = (X, \tau') \) is still metrizable. If moreover \( (X,\tau) \) is separable (respectively, analytic, or Polish), then the topology \( \tau' \) above can be chosen so that \( Z \) is zero-dimensional and separable (respectively, analytic, or Polish).
\end{lemma}

\begin{remark} \label{remfailure}
Assuming the Axiom of Choice, the fact that \( \mathcal{A} \) is countable is a necessary condition in 
Lemma~\ref{lemmachangetopology}, even in the case of a Polish space \( X \). To see this, set e.g.\ \( X = \pre{\omega}{\omega}	 \). Assume first that the Continuum Hypothesis \( \mathsf{CH} \) holds, and let \( \mathcal{A} = \{ \{ x 
\} \mid x \in \pre{\omega}{\omega} \} \), so that \( \mathcal{A} \) has cardinality \( \aleph_1 \). Then \( \mathcal{A} \subseteq \boldsymbol{\Pi}^0_1(\pre{\omega}{\omega}) \subseteq \boldsymbol{\Sigma}^0_\alpha(\pre{\omega}{\omega}) \) for every \( \alpha > 1 \). However, the unique 
topology \( \tau' \) on \( \pre{\omega}{\omega} \) such that \( \mathcal{A} 
\subseteq \tau' \) is the discrete topology, so that \( \tau' \nsubseteq 
\boldsymbol{\Sigma}^0_\alpha(\pre{\omega}{\omega}) \) for every \( \alpha 
> 1 \).
Assume now that \( \mathsf{CH} \) fails, and let \( A \subseteq \pre{\omega}{\omega} \) be any set of cardinality \( \aleph_1 \). Let \( \mathcal{A} = \{ \{ x \} \mid x \in A \} \), so that \( \mathcal{A} \subseteq \boldsymbol{\Pi}^0_1(\pre{\omega}{\omega}) \subseteq \boldsymbol{\Sigma}^0_\alpha(\pre{\omega}{\omega}) \) for every \( \alpha > 1 \) again. If \( \tau' \) is a topology such that \( \mathcal{A} \subseteq \tau' \), then \( A \in \tau' \). But then \( \tau' \nsubseteq \boldsymbol{\Sigma}^0_\alpha(\pre{\omega}{\omega}) \) (for every \( \alpha > 1 \)) because every Borel subset of \( \pre{\omega}{\omega} \) is either countable or has cardinality \( 2^{\aleph_0} > \aleph_1 \)  by e.g.\ \cite[Theorem 13.6]{kec}, and hence \( A \) cannot be Borel.
\end{remark}

\section{The structure of finite level Borel classes} \label{secfinitelevel}

The  notion of continuous function can be generalized in two natural ways. Given \(  \alpha \geq 1\), a function \( f \colon X \to Y \) between the topological spaces \(X,Y \) is called:
\begin{enumerate-(a)}
\item \label{item1}
\emph{\( \boldsymbol{\Sigma}^0_\alpha \)-measurable} if the preimage of each open set of \( Y \) is in \( \boldsymbol{\Sigma}^0_\alpha(X) \);
\item \label{item2}
\emph{\( \boldsymbol{\Delta}^0_\alpha \)-function} if the preimage of every set in \( \boldsymbol{\Sigma}^0_\alpha(Y) \) still belongs to \( \boldsymbol{\Sigma}^0_\alpha(X) \).
\end{enumerate-(a)}
In both cases, taking \( \alpha = 1 \) we get the class of continuous functions, 
and letting \( \alpha \) vary on the nonzero countable ordinals we get two 
interesting and natural stratifications of the class of all Borel 
functions which are intimately related (see  
\cite{motnewcharacterization}). In fact, although they give a probably less known stratification of 
the Borel functions (compared to the classical notions of \( 
\boldsymbol{\Sigma}^0_\alpha \)-measurable or Baire class \(\alpha\) 
functions), the classes of \( \boldsymbol{\Delta}^0_\alpha \)-functions are 
quite useful: for example, being closed under composition, they provide more 
natural reducibilities between subsets of Polish 
spaces~\cite{motamenable}, and they sometimes allow to transfer topological 
results from one space to another --- see e.g.\ \cite[Corollaries 4.25 and 4.26]
{motschsel}. In Section~\ref{sectionBanach} we will also present an application to Banach spaces whose proof 
involves these classes of functions.

The following definition further generalizes 
both~\ref{item1} and~\ref{item2}.

\begin{definition} \label{defSigma}
Let \( X,Y \) be topological spaces. For \( 1 \leq \beta \leq \alpha \) we let \( \Sigma_{\alpha,\beta}(X,Y) \) be the collection of all \emph{\( \Sigma_{\alpha,\beta} \) functions} from \( X \) to \( Y \), i.e.\ of those \( f \colon X \to Y \) such that \( f^{-1}(S) \in \boldsymbol{\Sigma}^0_\alpha(X) \) for every \( S \in \boldsymbol{\Sigma}^0_\beta(Y) \). To simplify the notation, we put \( \Sigma_\alpha(X,Y) = \Sigma_{\alpha,1}(X,Y) \) and \( \D_\alpha(X,Y) = \Sigma_{\alpha,\alpha}(X,Y) \), while
\( \mathsf{Bor}(X,Y) = \bigcup_{\alpha \geq 1} \D_\alpha(X,Y) = \bigcup_{\alpha \geq 1} \Sigma_\alpha(X,Y)  = \bigcup_{\alpha, \beta \geq 1} \Sigma_{\alpha,\beta}(X,Y) \) denotes the class of all \emph{Borel (measurable) functions} from \( X \) to \( Y \).
\end{definition}

\noindent
Thus \( \Sigma_\alpha(X,Y) \) and \( \D_\alpha(X,Y) \) are, respectively,  the collection of all \( \boldsymbol{\Sigma}^0_\alpha \)-measurable functions and all \( \boldsymbol{\Delta}^0_\alpha  \)-functions from \( X \) to \( Y \) considered in~\ref{item1} and~\ref{item2}. 
For simplicity of notation, we let \( \Sigma_{< \alpha,\beta}(X,Y) = 
\bigcup_{\gamma< \alpha} \Sigma_{\gamma,\beta} (X,Y) \); the classes \( 
\Sigma_{\alpha,<\beta}(X,Y) \), \( \Sigma_{< \alpha,<\beta}(X,Y) \), \( 
\Sigma_{<\alpha}(X,Y) \), and \( \D_{<\alpha}(X,Y) \) are defined similarly. 
Notice that if \( \alpha \) is limit, then the class \( \Sigma_{< \alpha, \beta}
(X,Y) \) is in general properly contained in the class of those \( f \colon X \to 
Y \) such that \( f^{-1}(S) \in \bigcup_{\gamma < \alpha} 
\boldsymbol{\Sigma}^0_\gamma(X) \) for every \( S \in 
\boldsymbol{\Sigma}^0_\beta(Y) \). Similar observations hold for the other 
mentioned classes.

\begin{definition} \label{deffinitelevel}
A class of the form \( \Sigma_{m,n} (X,Y) \) for \( 1 \leq n \leq m < \omega \) is called \emph{finite level Borel class}, and the functions in \( \bigcup_{1 \leq n \leq m < \omega} \Sigma_{m,n}(X,Y)  = \bigcup_{1 \leq m < \omega } \Sigma_m(X,Y) \) are called \emph{finite level Borel fuctions}.
\end{definition}

Note that by the observation following Definition~\ref{def:ctblpower}, \( f \colon X \to Y \) is a finite level Borel function if and only if \( f_\omega \colon \pre{\omega}{X} \to \pre{\omega}{Y} \) is a finite level Borel function.

We will now list some basic properties of the classes \( \Sigma_{\alpha,\beta}(X,Y) \) which will be tacitly used throughout the paper, leaving to the reader as an easy exercise to check their validity. Let \( X,Y \) be topological spaces and \( 1 \leq \beta \leq \alpha \).
 Obviously, if \( Y \subseteq Y' \) then \( 
\Sigma_{\alpha,\beta}(X,Y) \subseteq \Sigma_{\alpha,\beta}(X,Y') \) (in fact, \( \Sigma_{\alpha,\beta}(X,Y) = \{ f \in \Sigma_{\alpha,\beta}(X,Y') \mid \range(f) \subseteq Y \} \)), and if \( f \in \Sigma_{\alpha,\beta}(X,Y) \) and \( X' \subseteq X \) then \( f \restriction X' \in \Sigma_{\alpha,\beta}(X',Y) \). 
Moreover, if \( f \colon X \to Y \), 
\( g \colon X' \to Y' \), and \( f \sqsubseteq g \), then \( g \in \Sigma_{\alpha, 
\beta}(X',Y') \) implies \( f \in \Sigma_{\alpha,\beta}(X,Y) \). 
Arguing by induction on \( \gamma  \), for every \( \gamma < \omega_1 \) we have \( \Sigma_{\alpha, \beta}(X,Y) \subseteq \Sigma_{\alpha+\gamma,\beta+\gamma}(X,Y) \). Hence, in particular,  
 \( \Sigma_\alpha(X,Y) \subseteq \Sigma_{\alpha+\gamma,1+\gamma}(X,Y) \) for every \( \gamma < \omega_1 \),
 and \( \D_\alpha(X,Y) \subseteq \D_{\alpha'}(X,Y) \) for every \( 1 \leq \alpha \leq 
\alpha' \). Moreover, \( \D_\alpha(X,Y) \subseteq \Sigma_\alpha(X,Y) \) and, 
more generally \( \Sigma_{\alpha,\beta}(X,Y) \subseteq \Sigma_{\alpha', 
\beta'}(X,Y) \) whenever \( 1 \leq \beta' \leq \beta \leq \alpha \leq \alpha' \). Conversely, \( 
\Sigma_\alpha(X,Y) \subseteq \D_{\alpha \cdot \omega}(X,Y) \), where \( 
\alpha \cdot \omega \) is the first additively closed ordinal above \( \alpha \), 
and the index \( \alpha \cdot \omega \) cannot be lowered (see Corollary 
\ref{cor3variant}). Finally, we have the following simple property, which will be repeatedly used in the results below.

\begin{fact} \label{fct:basicSigmaalphabeta}
Let \( X,Y \) be arbitrary topological spaces, and let \( 1 \leq \beta \leq \alpha \). If \( \seq{ X_n }{ n \in \omega } \) is a partition of \( X \) into \( \boldsymbol{\Delta}^0_\alpha(X) \)-pieces and \( f_n \in \Sigma_{\alpha,\beta}(X_n,Y) \) for each \( n \in \omega \), then \( f = \bigcup_{n \in \omega} f_n \in \Sigma_{\alpha,\beta}(X,Y) \).
\end{fact}

The following lemma 
gives a sufficient condition for inclusion between the classes of functions introduced in Definition~\ref{defSigma}.

\begin{lemma} \label{lemma1}
Let \( X,Y \) be topological spaces. Let \( 1 \leq  \beta \leq \alpha \) and \( 1 \leq \beta' \leq \alpha' \). If \( \alpha \leq \alpha' \) and  \( \beta' \dotdiv \beta \leq \alpha' \dotdiv \alpha \), then \( \Sigma_{\alpha,\beta}(X,Y) \subseteq \Sigma_{\alpha',\beta'}(X,Y) \). 
\end{lemma}

\begin{proof}
If \( \beta' \leq \beta \) the result is obvious (see the properties listed above), so we can
 assume \( \beta < \beta' \). This implies \( \alpha < \alpha' \) because \( 0 < \beta' \dotdiv \beta \leq \alpha' \dotdiv \alpha \). Since
 \( \beta + (\beta' \dotdiv \beta) = \beta' \) (by \( \beta < \beta' \)) and
 \( \alpha + (\beta' \dotdiv \beta) \leq \alpha + (\alpha' \dotdiv \alpha) = \alpha' \) (by our hypotheses),
we have 
\[ 
\Sigma_{\alpha,\beta}(X,Y) \subseteq \Sigma_{\alpha+(\beta' \dotdiv \beta), \beta+(\beta' \dotdiv \beta)}(X,Y) = \Sigma_{\alpha+(\beta' \dotdiv \beta),\beta'}(X,Y) \subseteq \Sigma_{\alpha', \beta'}. \qedhere
\]
\end{proof}

The next lemma shows that the condition \( \alpha \leq \alpha' \) in the previous statement is necessary.

\begin{lemma}\label{lemma2}
Let \( X \) be an uncountable analytic space and \( Y \) be a nontrivial topological space.
For every \( 1 \leq \beta \leq \alpha \),
\[ 
\Sigma_{\alpha,\beta}(X,Y) \setminus \bigcup \{ \Sigma_{\alpha',\beta'}(X,Y) \mid 1 \leq \beta' \leq \alpha' < \alpha \} \neq \emptyset.
 \] 
Hence, in particular, \( \Sigma_{\alpha,\beta}(X,Y) \nsubseteq \Sigma_{\alpha',\beta'}(X,Y) \) whenever \( 1 \leq \beta' \leq \alpha' < \alpha \).
\end{lemma}

\begin{proof}
Let \(A \in  \boldsymbol{\Delta}^0_\alpha(X) \setminus \bigcup_{\alpha' < 
\alpha} \boldsymbol{\Sigma}^0_{\alpha'}(X) \), and let \( \emptyset \neq U \subsetneq Y \) be a nontrivial open set. 
Then fix \( y_0 \in U \) and \( y_1 \in Y \setminus U \), and set \( f(x) = y_0 \) if \( x 
\in A \) and \( f(x) = y_1 \) otherwise. Clearly \( f \in \Sigma_{\alpha, \beta}
(X,Y) \) (in fact, \( f^{-1}(S) \in \{ X, \emptyset, A , X \setminus A\} \subseteq \boldsymbol{\Delta}^0_\alpha(X) \subseteq 
\boldsymbol{\Sigma}^0_\alpha(X) \) for \emph{every} \( S \subseteq Y \)),
and since \( f^{-1}(U) = A \) and \( U \in 
\boldsymbol{\Sigma}^0_1(Y) \subseteq \boldsymbol{\Sigma}^0_{\beta'}(Y) 
\) for every \( \beta' \geq 1 \), \( f \notin \Sigma_{\alpha',\beta'}(X,Y) \) for every \( 1 \leq \beta' \leq \alpha' < \alpha \) by our choice of \( A \).
\end{proof}

Currently, we do not know if the condition \( \beta' \dotdiv \beta \leq \alpha' \dotdiv \alpha \) of Lemma \ref{lemma1} is necessary too, but since if \( 1 \leq \alpha 
\leq \alpha' < \omega \) then \( \beta' \dotdiv \beta \leq \alpha' \dotdiv \alpha  
\iff \alpha - \beta \leq \alpha' - \beta ' \), Theorem \ref{theorinclusion}\ref{theorinclusion-2} will show 
that this is the case for all finite levels. To prove this, we first need to consider 
some preliminary results. First of all,
arguing as in \cite[Lemma 6.5]{motbai}, it is easy to prove the following result.

\begin{lemma}\label{lemmaold}
Let \( X \) be a topological space, \( \alpha \geq 1 \), and for each \( n \in \omega \), let \( A_n  \in  \boldsymbol{\Sigma}^0_\alpha(\pre{\omega}{X}) \). Then
\[ 
A = \{ x \in \pre{\omega}{X} \mid \exists n\,  ( \pi^X_n(x) \in A_n) \} \in \boldsymbol{\Sigma}^0_\alpha (\pre{\omega}{X}).
 \] 
If moreover there is an increasing sequence  \( \seq{ \alpha_n }{ n \in \omega } \) of ordinals smaller than \( \alpha \) and cofinal in it such that 
each \( A_n \) is \( \boldsymbol{\Pi}^0_{\alpha_n}(\pre{\omega}{X}) \)-hard, then \( A \) is also 
\( \boldsymbol{\Sigma}^0_\alpha(\pre{\omega}{X}) \)-hard, i.e.\ a \( \boldsymbol{\Sigma}^0_\alpha(\pre{\omega}{X}) \)-complete set (and 
hence also a proper \( \boldsymbol{\Sigma}^0_\alpha(\pre{\omega}{X}) \) set).
\end{lemma}

\begin{definition} \label{defproperfunction}
Let \( X,Y \) be topological spaces and \( 1 \leq \beta \leq \alpha \). A function \( f \colon X \to Y \) is a \emph{proper \( \Sigma_{\alpha,\beta} \) function} if \( f \in \Sigma_{\alpha, \beta}(X,Y) \) and there is \( S \in \boldsymbol{\Sigma}^0_\beta(Y) \) such that \( f^{-1}(S) \) is a proper \( \boldsymbol{\Sigma}^0_\alpha(X) \) set. For ease of notation, when \( \beta = 1 \) in the definition above we simply say that \( f \) is a \emph{proper \( \Sigma_\alpha \) function}.
\end{definition}

It is easy to check that properness is preserved \( \sqsubseteq \)-upwards inside each \( \Sigma_{\alpha,\beta} \):

\begin{lemma} \label{lemmaproperpreserved}
Let \( X,Y,X',Y' \) be topological spaces, \( 1 \leq \beta \leq \alpha \), \( f \in \Sigma_{\alpha,\beta}(X,Y) \), and \( g \in \Sigma_{\alpha, \beta}(X',Y') \). If \( f \) is a proper \( \Sigma_{\alpha,\beta} \) function and \( f \sqsubseteq g \), then \( g \) is a proper \( \Sigma_{\alpha,\beta} \) function as well.
\end{lemma}

The next lemma is essentially a generalized version of~\cite[Proposition 6.6]{motbai}.

\begin{lemma} \label{lemma3}
Let \( X \) be a Polish space and \( Y \) be an arbitrary topological space. Let also \( \alpha > 1 \) and
\( f \colon X \to Y \) be such that 
\( f \in \Sigma_\alpha(X,Y)  \setminus  \Sigma_{<\alpha}(X,Y) \). Then \( f_\omega \colon \pre{\omega}{X} \to \pre{\omega}{Y}  \) is 
a proper \( \Sigma_{\alpha+\gamma, 1+ \gamma} \) function for every 
\( \gamma < \omega_1 \).
The same conclusion holds with \( \alpha = 1 \) whenever \( f \colon X \to Y \) is a non constant continuous function.
\end{lemma}

Notice that in fact \( f \in \Sigma_\alpha(X,Y)  \setminus  \Sigma_{<\alpha}(X,Y) \) is a strictly weaker condition than that of \( f \) being a proper \( \Sigma_{\alpha} \) function: for example, if \( \alpha = \beta+1 \) then the characteristic function of a proper \( \boldsymbol{\Pi}^0_\beta(X) \) set is in \( \Sigma_\alpha(X,2)  \setminus  \Sigma_{<\alpha}(X,2) \) but is not \( \Sigma_\alpha \) proper.

\begin{proof}
Consider first the case \( \alpha > 1 \), and argue by induction on \( \gamma < \omega_1 \). Assume \( \gamma = 0 \).
Given \( k \in \omega \) and an \emph{arbitrary} open set \( V \subseteq Y \), let \( U_{k,V} = \{ \seq{ x_n }{ n \in \omega } \in \pre{\omega}{X} \mid x_k \in V \} \). Notice that each open set \( U \subseteq \pre{\omega}{Y} \) is generated by sets of the form \( U_{k,V} \) using only finite intersections and countable unions (independently of whether \( Y \) is second-countable or not):
therefore it is enough to check that \( f_\omega^{-1}(U_{k,V}) \in \boldsymbol{\Sigma}^0_\alpha(\pre{\omega}{X}) \). By definition of \( f_\omega \),  
\[ 
f_\omega^{-1}(U_{k,V}) = \{ \seq{ x_n }{ n \in \omega } \in \pre{\omega}{X} \mid x_k \in f^{-1}(V) \},
 \] 
which is in \( \boldsymbol{\Sigma}^0_\alpha(\pre{\omega}{X}) \) since \( f \in \Sigma_\alpha(X,Y) \) and the projection \( \seq{ x_n }{ n \in \omega } \mapsto x_k \) is continuous. Therefore \( f_\omega \in \Sigma_\alpha(\pre{\omega}{X},\pre{\omega}{Y}) \).
Now fix an increasing sequence \( \seq{ \alpha_n }{ n \in \omega } \) of ordinals smaller than \( \alpha \) and cofinal in it (if \( \alpha = \beta + 1\) we can e.g.\ take \( \alpha_n = \beta \) for every \( n \in \omega \)).
By our hypothesis on \( f \) and Fact~\ref{factWadge}, there are open \( 
U_n \subseteq Y \) such that \( f^{-1}(U_n) \) is \( 
\boldsymbol{\Pi}^0_{\alpha_n}(X) \)-hard for every \( n \in \omega \): arguing as in~\cite[Lemma 6.5]{motbai}, it is then immediate to check that  \( 
V = \{ \seq{ y_n }{ n \in \omega } \in \pre{\omega}{Y} \mid \exists n \, (y_n \in U_n)  \} \) is an open 
set and that \( f_\omega^{-1}(V) = \{ \seq{ x_n }{ n \in \omega } \in \pre{\omega}{X} \mid \exists n \, (x_n \in f^{-1}(U_n)) \} \) is 
a  \(\boldsymbol{\Sigma}^0_\alpha(\pre{\omega}{X})\)-complete set, hence also a proper \( \boldsymbol{\Sigma}^0_\alpha(\pre{\omega}{X}) \) set. This shows that \( f_\omega \) is a proper \( \Sigma_\alpha  \) function.

Assume now \( \gamma >0 
\). Since \( f_\omega \in \Sigma_\alpha(\pre{\omega}{X},\pre{\omega}{Y}) \subseteq \Sigma_{\alpha+ \gamma, 1 + \gamma}(\pre{\omega}{X},\pre{\omega}{Y}) \), it remains to show that \( f_\omega^{-1}(B) \) is a proper \( \boldsymbol{\Sigma}^0_{\alpha+\gamma}(\pre{\omega}{X}) \) set for some \( B \in \boldsymbol{\Sigma}^0_{1+\gamma}(\pre{\omega}{Y}) \). Let \( \seq{ \gamma_n }{ n \in \omega } \) be an increasing 
sequence of ordinals smaller than \( \gamma \) and cofinal in it. By induction hypothesis, let \( A_n \in \boldsymbol{\Sigma}^0_{1+\gamma_n}(\pre{\omega}{Y})  \) be such that \( 
f_\omega^{-1}(A_n) \) is a proper \( 
\boldsymbol{\Sigma}^0_{\alpha+\gamma_n}(\pre{\omega}{X}) \) set, and let \( B_n = 
\pre{\omega}{Y} \setminus A_n \). Let \( B = \{ y \in \pre{\omega}{Y} \mid 
\exists n \, (\pi^Y_n(y) \in B_n ) \} \). By Lemma \ref{lemmaold}, \( B \in 
\boldsymbol{\Sigma}^0_{1+\gamma}(\pre{\omega}{Y}) \). Since \( f_\omega^{-1}(B_n) \) is a proper \( 
\boldsymbol{\Pi}^0_{\alpha+\gamma_n}(\pre{\omega}{X}) \) set (hence also \( 
\boldsymbol{\Pi}^0_{\alpha+\gamma_n}(\pre{\omega}{X}) \)-hard by Fact~\ref{factWadge}), \( f_\omega^{-1}(B)  = \{ x \in \pre{\omega}{X} \mid \exists n \, (\pi^X_n(x) \in f_\omega^{-1}(B_n)) \} \) is a \( 
\boldsymbol{\Sigma}^0_{\alpha+\gamma}(\pre{\omega}{X}) \)-complete (hence also \( \boldsymbol{\Sigma}^0_{\alpha+\gamma}(\pre{\omega}{X}) \) proper) set by Lemma \ref{lemmaold} again, hence \( f_\omega \) is a proper \( \Sigma_{\alpha+\gamma,1+\gamma} \) function.

The case \( \alpha = 1 \) can be treated in a similar way, using the fact that if \( f \colon X \to Y \) is continuous and non constant then there is an open \( U \subseteq Y \) such that \( f^{-1}(U) \) is neither empty nor the entire \( X \).	 
\end{proof}

Since the characteristic functions \( \chi_{C_n} \) used in Definition~\ref{defPn} to construct the generalized Pawlikowski functions are clearly in \( \Sigma_{n+1}(\pre{\omega}{2},2) \setminus \Sigma_n(\pre{\omega}{2},2) \), Lemma~\ref{lemma3} yields to the following corollary.

\begin{corollary} \label{corPn}
For every \( 1 \leq n < \omega \) and every \( \gamma < \omega_1 \), \( P_n \colon \pre{\omega}{(\pre{\omega}{2})} \to \pre{\omega}{2}\) is a proper \( \Sigma_{n+1+\gamma,1+\gamma} \) function. Moreover, the original Pawlikowski function \( P \colon \pre{\omega}{(\omega+1)} \to \pre{\omega}{\omega}\) is a proper \( \Sigma_{2+\gamma,1+\gamma} \) function (for every \( \gamma < \omega_1 \)).
\end{corollary}

\begin{corollary}\label{cor3}
Let \( X,Y \) be two uncountable analytic spaces, and let 
\( \alpha > 1 \). Then there is a function \( h \colon X \to Y \) such that  \( h \) 
is a proper \( \Sigma_{\alpha+\gamma,1+\gamma} \) function for every 
\( \gamma < \omega_1 \).
The same conclusion holds also for \( \alpha = 1 \) if either \( X \) is 
zero-dimensional or \( Y \) contains a homeomorphic copy of \( \RR \).
\end{corollary}

\begin{proof} 
If \( \alpha \geq 1 \), \( X = \pre{\omega}{(\pre{\omega}{2})} \), and \( Y = \pre{\omega}{2} \), then by Lemma \ref{lemma3} it is enough to let \( h = (\chi_A)_\omega \), where \( \chi_A \) is the characteristic function of some \( A \in \boldsymbol{\Delta}^0_\alpha(\pre{\omega}{2}) \setminus \bigcup_{\beta<\alpha} (\boldsymbol{\Sigma}^0_\beta(\pre{\omega}{2}) \cup \boldsymbol{\Pi}^0_\beta(\pre{\omega}{2})) \) (where we set \( \boldsymbol{\Sigma}^0_0(\pre{\omega}{2}) = \{ \emptyset \} \) and \( \boldsymbol{\Pi}^0_0(\pre{\omega}{2}) = \{ \pre{\omega}{2} \} \)).

Now assume that \( X,Y \) are arbitrary uncountable analytic spaces, let \( C_X \subseteq X \), \( C_Y \subseteq Y \) be homeomorphic copies of \( \pre{\omega}{2} \), and fix \( \alpha \geq 1 \). Since \( \pre{\omega}{(\pre{\omega}{2})} \) and \( \pre{\omega}{2} \) are homeomorphic, by the previous paragraph there is \( h' \colon  C_X \to C_Y \) which is a proper \( \Sigma_{\alpha + \gamma, 1 + \gamma} \) function for every \( \gamma < \omega_1 \): therefore, by e.g.\ Lemma~\ref{lemmaproperpreserved} it is enough to extend \( h' \) to some \( h \in \Sigma_\alpha(X,Y) \). If \( \alpha > 1 \), simply extend \( h' \) by an arbitrary constant function on \( X \setminus C_X \) (the resulting \( h \) is still a \( \Sigma_\alpha \) function because \( C_X \) is closed in \( X \)). If instead \( \alpha = 1 \), then \( h' \) is a continuous function defined on a closed subset of \( X \). If \( X \) is zero-dimensional, then \( h' \) can be extended to a continuous  \( h \colon X \to Y \) by~\cite[Theorem 7.3]{kec}. Similarly, if \( Y \) contains a homeomorphic copy \( Y' \) of \( \RR \), then we can assume without loss of generality that \( C_Y \subseteq Y' \), so that \( h' \) can be extended to a continuous \( h \colon X \to Y' \subseteq Y \) by ~\cite[Theorem 1.3]{kec}.
\end{proof}

In particular, setting \( \alpha = m-n+1 \) and \( \gamma = n-1 \) in Corollary~\ref{cor3} we get that for every \( 1 \leq n \leq m < \omega \) there exists a \emph{proper} \( \Sigma_{m,n} \) function between any pair of uncountable analytic spaces \( X \), \( Y \). A much stronger result will be proved in Theorem~\ref{theorinclusion}\ref{theorinclusion-1}. The next corollary shows that the smallest \( \beta \) for which \( \Sigma_\alpha (X,Y) \subseteq \D_\beta(X,Y) \) is \( \beta = \alpha \cdot \omega \).

\begin{corollary}\label{cor3variant}
Let \( X,Y \) be uncountable analytic spaces. For every 
\( \alpha > 1 \) there is a function \( h \colon X \to Y \) such that 
\( h \in \Sigma_\alpha (X,Y) \setminus \D_{< \alpha \cdot \omega}(X,Y) \).
\end{corollary}

\begin{proof}
Let \( h \) be any function satisfying the conclusion of Corollary \ref{cor3}, 
so that \( h \in \Sigma_ \alpha(X,Y) \): 
since, \( \D_\beta(X,Y) \subseteq 
\D_{\beta'}(X,Y) \) whenever \( \beta \leq \beta' \), it is clearly enough to 
show that \( h \notin \D_{\alpha \cdot n}(X,Y) \) for every \( n \in \omega \). 
Fix \( n \in \omega \) and let \( \gamma = \alpha \cdot n \): by Corollary 
\ref{cor3}, \( h \) is a proper \( \Sigma_{\alpha+\alpha \cdot n, 1 + \alpha 
\cdot n} \) function, and hence \( h \notin \D_{\alpha \cdot n} \)	since \( 
\alpha > 1 \) implies \( \alpha + \alpha \cdot n = \alpha \cdot (n+1) > 
1+\alpha \cdot n \).
\end{proof}

The next theorem, which is the main result of this section, completely 
determines the structure of finite level Borel classes under inclusion in the context of uncountable analytic spaces. In 
particular, it shows 
that all the finite level Borel classes \( \Sigma_{m,n}(X,Y) \) are distinct, and 
that in fact
in each of these classes there are genuinely new 
functions (i.e.\ \( \Sigma_{m,n}(X,Y) \) is not covered by the collection of 
all finite level Borel classes which do not already contain it). 

\begin{theorem} \label{theorinclusion}
Let \( X,Y \) be uncountable analytic spaces and \( 1 \leq n \leq m < \omega \). 
\begin{enumerate-(a)}
\item \label{theorinclusion-1}
There is \( f \in \Sigma_{m,n}(X,Y) \) such that \( f \notin \Sigma_{m',n'}(X,Y) \) for every \( 1 \leq n' \leq m' < \omega \) satisfying (at least one of) \( m' < m \) or \( m' - n' < m-n \).
\item \label{theorinclusion-2}
Let 
\( 1 \leq n' \leq m' < \omega \). Then 
\[ 
\Sigma_{m,n}(X,Y) \subseteq \Sigma_{m',n'}(X,Y) \iff m \leq m' \wedge m-n \leq m'-n'. 
\] 
In particular,  if \( m \neq m' \) or \( n \neq n' \) then \( \Sigma_{m,n}(X,Y) \neq \Sigma_{m',n'}(X,Y) \).
\end{enumerate-(a)}
\end{theorem}

Notice that any \( f \) as in Theorem~\ref{theorinclusion}~\ref{theorinclusion-1} is necessarily a proper \( \Sigma_{m,n} \) function, as otherwise we would have \( f \in \Sigma_{m',n'}(X,Y) \) for \( m' = m \) and \( n' = n+1 \) (contradicting the choice of \( f \)).

\begin{proof}
First we show~\ref{theorinclusion-1}. If \( m = 1 \), then necessarily \( m= n 
= 1 \) and thus neither \( m' < m \) nor \( m' - n' < m-n \) can be realized by 
any \( 1 \leq n' \leq m' < \omega \). Therefore in this case~\ref{theorinclusion-1} 
amounts to show that \( \Sigma_{1,1}(X,Y) \), the class of continuous 
functions, is nonempty, which is trivially true. If \( m= n > 1 \) and \( 1 \leq n' 
\leq m' < \omega \), only the inequality \( m' < m \) can be realized, so~\ref{theorinclusion-1} holds by Lemma~\ref{lemma2}. Finally, assume \( m > n \geq 1 \). Let \( X_0 \) be an uncountable closed subset of \( X \) such that \( X_1 = X \setminus X_0 \) is uncountable too. Apply Lemma~\ref{lemma2} 
to get \( f_0 \in \Sigma_{m,n}(X_0, Y) \) such that \( f_0 \notin 
\Sigma_{m',n'}(X_0,Y) \) for every \( 1 \leq n' \leq m' < m \), and apply Corollary 
\ref{cor3} with \( \alpha = m-n+1 > m-n \geq 1 \) to get \( f_1  \colon X_1 \to Y \) such that \( f_1 \) is a proper \( \Sigma_{m-
n+1+\gamma,1+\gamma} \) function for every \( \gamma < \omega_1 \). 
Note that taking \( \gamma = n-1 \) we get \( f_1 \in \Sigma_{m,n}(X_1,Y) 
\), while taking \( \gamma = n'-1 \) we get that \( f_1 \) is a proper \( \Sigma_{m-n+n',n'} \) function, so that \( f_1 \notin \Sigma_{m',n'}(X_1,Y) \) for every \( 1 \leq n' \leq m'< \omega \) such that \( m' - n' < m-n \) (because in such case \( m' < m-n+n' \)).
Let now \( f = f_0 \cup f_1 \): then \( f \in 
\Sigma_{m,n}(X,Y) \) by Fact~\ref{fct:basicSigmaalphabeta}, which can be applied because \( m > 1 \) by case assumption (so that \( X_0, X_1 \in \boldsymbol{\Delta}^0_m \)). Finally, if \( f 
\in \Sigma_{m',n'}(X,Y) \) for some \( 1 \leq n' \leq m' < m \) then \( f_0 
= f \restriction X_0 \in \Sigma_{m',n'}(X_0,Y) \) too, contradicting the choice of \( f_0 \), while if \( f \in 
\Sigma_{m',n'}(X,Y) \) for some  \( 1 \leq n' \leq m' < \omega \) such that \( m' - n' 
< m-n \) then \( f_1 = f \restriction X_1 \in \Sigma_{m',n'}
(X_1,Y) \), contradicting the choice of \( f_1 \): therefore \( f \) is as required.

We now show~\ref{theorinclusion-2}.
Since~\ref{theorinclusion-1} implies that if either \( m' < m \) or \( m' - n' < m-n \) then \( \Sigma_{m,n}(X,Y) \not\subseteq \Sigma_{m',n'}(X,Y) \), only
the direction from right to left need to be proved. 
We consider two cases: if \( n' \leq n \) then automatically \( n' \dotdiv n = 0 \leq m' \dotdiv m \). 
If instead \( n \leq n' \), we get that \( n' \dotdiv n  = n' - n \), and therefore, since also 
\( m' \dotdiv m  = m' - m \) by the assumption \( m \leq m' \), we get \( n' \dotdiv n \leq m' \dotdiv m \) because \( m-n  \leq m' - n'  \iff n' - n \leq m' - m \).
So in all cases \( n' \dotdiv n \leq m' \dotdiv m \), and since we are assuming also \( m \leq m' \), then \( \Sigma_{m,n}(X,Y) \subseteq \Sigma_{m',n'}(X,Y) \) by Lemma \ref{lemma1}.
\end{proof}

In particular, Theorem \ref{theorinclusion}\ref{theorinclusion-2} allows us  to extend the inclusion diagram of Figure~\ref{figinclusion} to all finite level Borel classes of functions defined between arbitrary uncountable analytic spaces. The picture one gets is shown in Figure~\ref{figdiagram}.
\begin{figure}[!htbp] 
\centering
\mbox{
\xymatrix{
& & & & \\
& & & \Sigma_{4,1}(X,Y) \ar@{.>}[ur] \ar@{.>}[dr] &\\
& & \Sigma_{3,1}(X,Y)  \ar@{->}[ur] \ar@{->}[dr] & & \\
& \Sigma_{2,1}(X,Y)   \ar@{->}[ur] \ar@{->}[dr] & & \Sigma_{4,2}(X,Y) \ar@{->}[uu] \ar@{.>}[ur] \ar@{.>}[dr] &\\
\Sigma_{1,1}(X,Y) \ar@{->}[ur] \ar@{->}[dr] & & \Sigma_{3,2}(X,Y) \ar@{->}[uu] \ar@{->}[ur] \ar@{->}[dr] & & \\
& \Sigma_{2,2}(X,Y) \ar@{->}[dr] \ar@{->}[uu] \ar@{->}[ur] & & \Sigma_{4,3}(X,Y) \ar@{->}[uu] \ar@{.>}[ur] \ar@{.>}[dr] & \\
& & \Sigma_{3,3}(X,Y) \ar@{->}[uu] \ar@{->}[ur] \ar@{->}[dr] & &  \\
& & & \Sigma_{4,4}(X,Y) \ar@{->}[uu] \ar@{.>}[ur] \ar@{.>}[dr] & \\
& & & & }
}
\caption{The inclusion diagram of all finite level Borel classes.}
\label{figdiagram}
\end{figure}
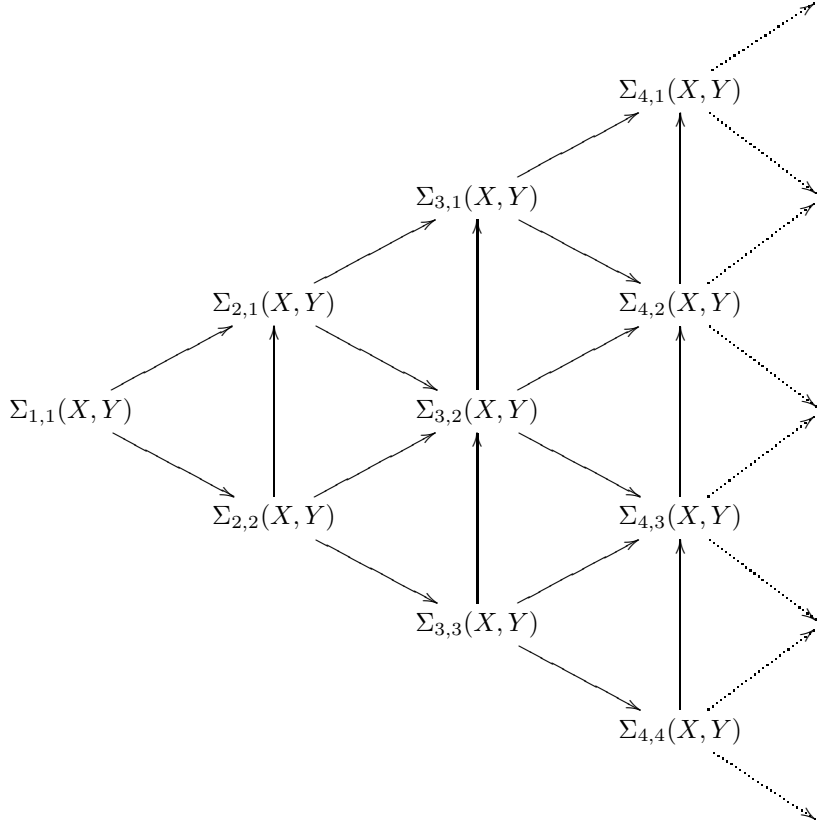

The use of countable powers of (finite level) Borel functions was a fundamental 
tool in proving Theorem~\ref{theorinclusion}, so it is quite natural to 
abstractly analyze this operation. Since we already noticed that continuous 
functions are closed under taking countable powers, a first question is to ask 
which other Borel classes are closed under such operation. Using the ideas 
involved in Lemmas~\ref{lemma2} and \ref{lemma3}, we get the following 
exhaustive answer.

\begin{proposition} \label{proppower} 
Let \( X \) be an uncountable analytic space and \( Y \) be a nontrivial topological space.
Then for every \( 1 \leq \beta  \leq \alpha \) the following are equivalent:
\begin{enumerate-(a)}
\item \label{proppower-1}
\( f \in \Sigma_{\alpha, \beta}(X,Y) \Rightarrow f_\omega \in \Sigma_{\alpha,\beta}(\pre{\omega}{X}, \pre{\omega}{Y}) \) for every \( f \colon X \to Y \);
\item \label{proppower-2}
\( \beta = 1 \);
\item \label{proppower-3}
\( f \in \Sigma_{\alpha, \beta}(X,Y) \iff f_\omega \in \Sigma_{\alpha,\beta}(\pre{\omega}{X}, \pre{\omega}{Y}) \) for every \( f \colon X \to Y \).
\end{enumerate-(a)}
\end{proposition}

\begin{proof}
The implication \ref{proppower-3}\( \Rightarrow \)\ref{proppower-1} is trivial, while \ref{proppower-2}\( \Rightarrow \)\ref{proppower-3} by the observation following Definition~\ref{def:ctblpower}.
To prove \ref{proppower-1}\( \Rightarrow \)\ref{proppower-2} it is enough to show that there is \( f \in \D_\alpha(X,Y) \) 
such that \( f_\omega \) is a proper \( \Sigma_{\alpha} \) function, because this implies that for every \( 1 < \beta \leq 
\alpha \),  \( f \in \Sigma_{\alpha, \beta}(X,Y) \) but \( f_\omega \notin \Sigma_{\alpha,2}(\pre{\omega}{X}, \pre{\omega}{Y}) \supseteq \Sigma_{\alpha,\beta}(\pre{\omega}{X}, 
\pre{\omega}{Y}) \). Let \( \emptyset, X \neq A \in \boldsymbol{\Delta}^0_\alpha(X) \) be \( \boldsymbol{\Pi}^0_\gamma(X) \)-hard for every 
\( 1 \leq \gamma < \alpha \). 
Pick distinct \( y_0, y_1 \in Y \) such that \( y_0 \in U \) and \( y_1 \notin U \) for some open 
\( U \subseteq Y \), and let \( f \colon X \to \{ y_0, y_1 \} \subseteq Y \) be defined by \( f(x)  = y_0 \iff x \in A \). Then \( f \in 
\D_\alpha(X,Y) \) (because the preimage of \emph{every} subset of \( Y \) is in \( \{ \emptyset, X, A, X \setminus A \} \subseteq \boldsymbol{\Delta}^0_\alpha(X) \)). 
Moreover, arguing as in~\cite[Lemma 6.5]{motbai} we get 
that \( V = \{ \seq{ y_n }{ n \in \omega } \in \pre{\omega}{Y} \mid \exists n \, (y_n \in U ) \}  \subseteq \pre{\omega}{Y} \) is open and that
\[
 f^{-1}_\omega(V) = \{ \seq{ x_n }{ n \in \omega } \in \pre{\omega}{X} \mid \exists n \, (x_n \in f^{-1}(U)) \} = \{ \seq{ x_n }{ n \in \omega } \in \pre{\omega}{X} \mid \exists n \, (x_n \in A ) \}
\] 
is a proper \( \boldsymbol{\Sigma}^0_\alpha(X) \) set, so that \( f_\omega \) is a proper \( 
\Sigma_\alpha \) function.\end{proof}

We end this section by remarking that Corollary~\ref{cor3}  gives also some 
partial results concerning infinite level Borel classes. For example, it implies that 
for all uncountable analytic spaces \( X,Y \) and every \( 1 < \beta \leq \alpha 
\), \( \Sigma_{\alpha,\beta}(X,Y) \) is properly included in \( 
\Sigma_\alpha(X,Y) \): in fact, there is \( f \in \Sigma_\alpha(X,Y) \) such that 
\( f^{-1}(S) \) is a proper \( \boldsymbol{\Sigma}^0_{\alpha+(\beta \dotdiv 
1)}(X) \) set for some \( S \in  \boldsymbol{\Sigma}^0_\beta(Y) \), whence \( f \notin \Sigma_{\alpha,\beta}(X,Y) \). However, 
it seems that we are still far from a full comprehension of the entire structure 
of (all) Borel classes of functions under inclusion. For example, the methods 
developed in this paper do not answer the following sample question:

\begin{question}
Let \( X, Y \) be uncountable analytic spaces, and let \( 1 < m < n < \omega \). Is the Borel class \( \Sigma_{\omega,n}(X,Y) \) properly included in \( \Sigma_{\omega,m}(X,Y) \)?
\end{question}

\section{\(\omega\)-decomposable functions} \label{sectiondecomposable}

\begin{definition}\label{defdec} 
Let \( X, Y \) be topological spaces and \( \beta \geq 1 \). A function \( f \colon X \to Y \) is said \emph{\(\omega\)-decomposable in \( \Sigma_\beta \) functions} if there is a countable partition \( \seq{ X_n }{ n \in \omega } \) of \( X \) such that \( f \restriction X_n \in \Sigma_\beta(X_n, Y) \) for every \( n \in \omega \). 

The class of all  functions from \( X \) to \( Y \) which are \( \omega \)-decomposable in \( \Sigma_\beta \) functions is denoted by \( \Dec(\beta)(X,Y) \), and for every collection of functions  \( \F(X,Y) \) from \( X \) to \( Y \) we denote by \( \F^{\Dec(\beta)}(X,Y) \) the class \( \F(X,Y) \cap \Dec(\beta)(X,Y) \). 
\end{definition}

When \( \beta = 1 \) in the previous definition, i.e.\ when \( f \colon X \to Y \) is \(\omega\)-decomposable in \( \Sigma_1 \) (hence continuous) functions, we just say that \( f \) is \emph{\(\omega\)-decomposable}, and we write \( \Dec(X,Y) \) and \( \F^\Dec(X,Y) \) instead of, respectively, \( \Dec(1)(X,Y) \) and \( \F^{\Dec(1)}(X,Y) \). We also set \( \Dec(< \beta)(X,Y) = \bigcup_{\gamma < \beta} \Dec(\gamma)(X,Y) \). Obviously, \( \Dec(\beta)(X,Y) \subseteq \Dec(\beta')(X,Y) \) for every \( 1 \leq \beta \leq \beta' \), and \( f \restriction X' \in \Dec(\beta)(X',Y) \) for every \( f \in \Dec(\beta)(X,Y) \) and  \( X' \subseteq X \).

With this notation, the generalized Lusin's question~\ref{questionLusin} can be reformulated as follows: given two uncountable Polish spaces and \( \beta \geq 1 \), does \( \mathsf{Bor}(X,Y) \) equal \( \mathsf{Bor}^{\Dec(\beta)}(X,Y) \)?

Using Kuratowski's lemma~\ref{lemmaextension}, it is easy to check that if \( X,Y \) are metrizable spaces with \( Y \) separable and \( f \in \Dec(\beta)(X,Y) \) is a \emph{Borel} function, then the partition \( \seq{X_n}{n \in \omega} \) in Definition~\ref{defdec} may be taken to consist of Borel sets (in fact, the assumption of separability of \( Y \) may be dropped in the special case \( \beta = 1 \)) --- see e.g.\ Proposition \ref{propdefinabledec}\ref{propdefinabledec-2}). This suggests to consider the
following definable (Borel) version of \(\omega\)-decomposability in \( \Sigma_\beta \) functions.

\begin{definition}\label{defdecalpha}
Let \(\alpha, \beta \geq 1 \) and \( X,Y \) be topological spaces. We denote by \( \D^\beta_\alpha(X,Y) \) the collection of all
\( f \colon X \to Y \) which are \emph{in \( \Sigma_\beta \) on a \( \boldsymbol{\Delta}^0_\alpha \)-partition}, i.e.\ such that there is a countable partition \( \seq{ X_n }{ n \in \omega } \) of \( X \) in \( \boldsymbol{\Delta}^0_\alpha(X) \)-pieces such that \( f \restriction X_n \in \Sigma_\beta(X_n,Y) \) for every \( n \in \omega \). To simplify the presentation, we will say that a partition \( \seq{ X_n }{ n \in \omega } \) as above \emph{witnesses \( f \in \D^\beta_\alpha(X,Y) \)}.  
\end{definition}

For \( \alpha, \beta \geq 1 \), we also let \( \D^\beta_{< \alpha}(X,Y) = \bigcup_{\gamma<\alpha} \D^\beta_\gamma(X,Y) \). Obviously, \( \Sigma_\beta(X,Y) \subseteq \D^\beta_\alpha(X,Y) \),  if \( Y \subseteq Y' \) then \( \D^\beta_\alpha(X,Y) \subseteq  \D^\beta_\alpha(X,Y') \) (in fact, \( \D^\beta_\alpha(X,Y) = \{ f \in \D^\beta_\alpha(X,Y') \mid \range(f) \subseteq Y \} \)), and if \( f \in \D^\beta_\alpha(X,Y) \) and \( X' \subseteq X \) then \( f \restriction X' \in \D^\beta_\alpha(X',Y) \).  
Using the notation of Definition~\ref{defdecalpha}, Conjecture~\ref{conjweakJR} can be restated as
\( \D_n(X,Y) = \D^1_n(X,Y) \) for every \( n \in \omega \),
while Conjecture~\ref{conjstrongJR} can be reformulated as
\( \Sigma_{m,n}(X,Y) = \D_m^{m-n+1}(X,Y) \) for every \(  1 \leq n \leq m < \omega \).

The next lemma shows, in particular, that the class \( \D^\beta_\alpha(X,Y) \) becomes (potentially) interesting only if \( \beta < \alpha \).

\begin{lemma} \label{lemmadefdecomposable}
Let \( \alpha, \beta \geq 1 \) and \( X, Y \) be topological spaces.
\begin{enumerate-(a)}
\item \label{lemmadefdecomposable-1}
If \( \beta \geq \alpha \), then \( \D^\beta_\alpha(X,Y) = \Sigma_\beta(X,Y) \). 
\item \label{lemmadefdecomposable-2}
If \( \beta < \alpha \), then \( \D^\beta_\alpha(X,Y) \subseteq \Sigma_{\alpha, 1 + (\alpha \dotdiv \beta)}(X,Y) \).
\end{enumerate-(a)}
Hence, in particular, \( \D^1_\alpha(X,Y) \subseteq \D_\alpha(X,Y) \).
\end{lemma}

\begin{proof}
Let \( \seq{ X_n }{ n \in \omega } \) witness \( f \in \D^\beta_\alpha(X,Y) \). Assume first \( \beta \geq \alpha \). Then \( X_n \in \boldsymbol{\Delta}^0_\alpha(X) \subseteq \boldsymbol{\Delta}^0_\beta(X) \) for every \( n \in \omega \), and hence \( f \in \Sigma_\beta(X,Y) \) by Fact~\ref{fct:basicSigmaalphabeta}: this shows \( \D^\beta_\alpha(X,Y) \subseteq \Sigma_\beta(X,Y) \), whence \( \D^\beta_\alpha(X,Y) = \Sigma_\beta(X,Y) \). Assume now \( \beta < \alpha \). Then for every \( n \in \omega \)
\[
 f \restriction X_n \in \Sigma_\beta(X_n,Y) \subseteq \Sigma_{\beta+(\alpha \dotdiv \beta), 1+ (\alpha \dotdiv \beta)}(X_n,Y) = \Sigma_{\alpha, 1 + (\alpha \dotdiv \beta)}(X_n,Y), 
\]
and hence \( f \in \Sigma_{\alpha,1+(\alpha \dotdiv \beta)}(X,Y) \) by Fact~\ref{fct:basicSigmaalphabeta} again: therefore \( \D^\beta_\alpha(X,Y) \subseteq \Sigma_{\alpha,1+(\alpha \dotdiv \beta)}(X,Y) \).
\end{proof}

Whether the reverse inclusion of Lemma \ref{lemmadefdecomposable}\ref{lemmadefdecomposable-2} holds for \( 1 \leq  \beta < \alpha < \omega \) is the content of the strong generalization of the Jayne-Rogers theorem (Conjecture~\ref{conjstrongJR}), while the weak generalization (Conjecture~\ref{conjweakJR}) corresponds to the case \( \beta = 1 \).

For the next proposition we will use the following well-known result (see e.g.~\cite[Chapter III \textsection 35.VI, p.~434]{kurbook}).

\begin{lemma}[Kuratowski]\label{lemmaextension}
Let \( X \) be a metrizable space, \( Y \) be a Polish space, and \( \alpha \geq 1 \). 
For every \( A \subseteq X \) and \( f \in \Sigma_\alpha(A,Y) \) there is a set \( A \subseteq A' \subseteq \cl(A) \)  and a function \( f' \in \Sigma_\alpha(A',Y) \) such that  \( A' \in \boldsymbol{\Pi}^0_{\alpha+1}(X) \) and \( f' \) extends \( f \). When \( \alpha = 1 \), the space \( Y \) may be assumed to be just completely metrizable.
\end{lemma}

Part~\ref{propdefinabledec-2} of Proposition~\ref{propdefinabledec} is implicit in~\cite{keldys}, but we will fully reprove it here for the reader's convenience.

\begin{proposition}\label{propdefinabledec}
Let \( X , Y\) be metrizable spaces with \( Y \) separable, \(  \alpha ,\beta \geq 1 \), and \( f \in \Dec(\alpha)(X,Y) \).
\begin{enumerate-(a)}
\item \label{propdefinabledec-1}
If \( \gr(f) \in \boldsymbol{\Sigma}^0_{\beta} (X \times \widetilde{Y}) \), then \( f \in \D^\alpha_\gamma(X,Y) \) with \( \gamma = \max \{ \alpha + 2, \alpha + (\beta \dotdiv 1)  \} \).
\item \label{propdefinabledec-2}
If \( f \in \Sigma_\beta(X,Y) \), then \( f \in \D^\alpha_{\beta + 1}(X,Y) \).
\end{enumerate-(a)}
When \( \alpha = 1 \), in all the results above the space \( Y \) may be assumed to be just metrizable.
\end{proposition}

\begin{proof}
Let \( \seq{ X_n }{ n \in \omega } \) be a countable partition of 
\( X \) such that \( f \restriction X_n \in \Sigma_\alpha(X_n,Y) \subseteq \Sigma_\alpha(X_n, \widetilde{Y})  \) for every 
\( n \in \omega \). Using 
Lemma~\ref{lemmaextension}, pick for every \( n \in \omega \) a 
\( \boldsymbol{\Pi}^0_{\alpha+1}(X) \) set 
\( X_n \subseteq A_n \subseteq \cl(X_n) \) and a function 
\( g_n \in \Sigma_\alpha(A_n, \widetilde{Y}) \) extending \( f \restriction X_n \). Let 
\( B_n = \{ x \in A_n \mid g_n(x)  = f(x) \} \): then 
\( X_n \subseteq B_n \subseteq A_n \) 
and  \( f \restriction B_n \in \Sigma_\alpha(B_n, \widetilde{Y}) \) (because 
\( f \restriction B_n  = g_n \restriction B_n \)), whence 
\( f \restriction B_n \in \Sigma_\alpha(B_n, Y) \) because \( \range(f) \subseteq Y \).

Assume first that \( \gr(f) \in \boldsymbol{\Sigma}^0_\beta(X \times \widetilde{Y} ) \). Notice 
that 
\[ 
B_n = \{ x \in X \mid x \in A_n \wedge (x,g_n(x)) \in \gr(f) \}, 
\] 
and that the 
map \( A_n \to X \times \widetilde{Y} \colon x \mapsto (x, g_n(x)) \) is in \( \Sigma_\alpha(A_n, X \times \widetilde{Y}) \).
Therefore 
\( B_n \) is the intersection of a \( \boldsymbol{\Pi}^0_{\alpha+1}(X) \) set with a \( \Sigma_{\alpha + (\beta \dotdiv 1)}(X) \) set, whence
\( B_n \in \boldsymbol{\Sigma}^0_{\gamma}(X) \) with 
\( \gamma = \max \{ \alpha+2, \alpha + (\beta \dotdiv 1) \} \). Since \( X_n \subseteq B_n \) implies that the family \( \seq{B_n}{n \in \omega} \) covers \( X \), using standard arguments we can refine it to a countable partition \( \seq{C_k}{k \in \omega} \) of \( X \) into \( \boldsymbol{\Delta}^0_\gamma(X) \)-pieces, and using the fact that  each \( C_k \) is a subset of some \( B_n \) it is straightforward to verify that such a partition witnesses \( f \in \D^\alpha_{\gamma}(X,Y) \).

Assume now that \( f \in \Sigma_\beta(X,Y) \). We may assume without loss of generality that \( \beta > \alpha \) (otherwise \( f \in \Sigma_\alpha(X,Y) \subseteq \D^\alpha_{\gamma} \) for every \( \gamma \geq 1 \)), so that the map \( A_n \to \widetilde{Y} \times \widetilde{Y} \colon x \mapsto (f(x),g_n(x)) \) is in \( \Sigma_\beta(A_n, \widetilde{Y} \times \widetilde{Y}) \). Since the diagonal of \( \widetilde{Y} \) is closed, we get that each \( B_n \) is the intersection of a \( \boldsymbol{\Pi}^0_{\alpha+1} (X) \) set with a \( \boldsymbol{\Pi}^0_{\beta}(X) \) set, whence \( B_n \in \boldsymbol{\Pi}^0_\beta(X) \subseteq \boldsymbol{\Delta}^0_{\beta+1}(X) \). Arguing as in the first part, the covering \( \seq{B_n}{n \in \omega} \) of \( X \) may be refined to a partition \( \seq{C_k}{k \in \omega} \) of \( X \) into \( \boldsymbol{\Delta}^0_{\beta+1}(X) \)-pieces which clearly witnesses \( f \in \D^\alpha_{\beta+1}(X,Y) \).

The additional part follows from the fact that when \( \alpha =1 \) Lemma~\ref{lemmaextension} may be applied in such wider context.
\end{proof}

\begin{corollary}\label{cor+1}
Let \( X,Y \) be metrizable spaces and \( \alpha \geq 1 \). For all \( f \in \Sigma_\alpha(X,Y) \), \( f \) is \(\omega\)-decomposable if and only if \( f \in \D^1_{\alpha+1}(X,Y) \). In particular, \( \D_\alpha^\Dec(X,Y) \subseteq \Sigma_\alpha^\Dec(X,Y) \subseteq \D^1_{\alpha+1}(X,Y) \).
\end{corollary}

\begin{theorem}\label{theor:luzin0}
Let \( X,Y \) be metrizable spaces with \( Y \) separable. 
\begin{enumerate-(a)}
\item \label{theor:luzin0-1}
For all \( \alpha \geq 1 \), if \( f \in \Sigma_{\alpha+1}(X,Y) \) is a proper \( \Sigma_{\alpha+3,3} \) function, then \( f \) is not \(\omega\)-decomposable in \( \Sigma_\alpha \) functions.   
\item \label{theor:luzin0-2}
For all \( 1 \leq m < n < \omega \), if \( f \in \Sigma_n (X,Y) \setminus \Sigma_{n+1, n-m+2}(X,Y) \), then \( f \) is not \(\omega\)-decomposable in \( \Sigma_m \) functions.
\end{enumerate-(a)}
Moreover, in~\ref{theor:luzin0-1} (respectively,~\ref{theor:luzin0-2})  when \( \alpha = 1 \) (respectively, \( m = 1 \)) we may assume \( Y \) to be just a metrizable space.
\end{theorem}

\begin{proof}
We first show~\ref{theor:luzin0-1}.
Assume towards a contradiction that \( f \) is \(\omega\)-decomposable in \( \Sigma_\alpha \) functions. Then by \( f \in \Sigma_{\alpha+1}(X,Y) \) and Proposition~\ref{propdefinabledec}\ref{propdefinabledec-2} we get \( f \in \D^\alpha_{\alpha+2}(X,Y) \),
whence \( f \in \Sigma_{\alpha+2,3}(X,Y) \)  by Lemma~\ref{lemmadefdecomposable}\ref{lemmadefdecomposable-2}: but this contradicts the assumption that \( f \) is a proper \( \Sigma_{\alpha+3,3} \) function, and therefore \( f \notin \Dec(\alpha)(X,Y) \).

Part~\ref{theor:luzin0-2} is similar. Assuming towards a contradiction that \( f \in \Dec(m)(X,Y) \), from \( f \in \Sigma_n(X,Y) \) and Proposition~\ref{propdefinabledec}\ref{propdefinabledec-2} we get \( f \in \D^m_{n+1}(X,Y) \), whence \( f \in \Sigma_{n+1, n-m+2} \) by Lemma~\ref{lemmadefdecomposable}\ref{lemmadefdecomposable-2}: since this contradicts our assumption on \( f \), we conclude that \( f \notin \Dec(m)(X,Y) \).
\end{proof}

By Corollary~\ref{corPn}, when \( \alpha = 1 \)
 we can let the function \( f \) in Theorem~\ref{theor:luzin0}\ref{theor:luzin0-1} be the 
Pawlikowsi function \( P \): in this way we get an 
alternative proof of the fact that \( P \) is not \(\omega\)-decomposable. The original proof was based on the observation that if \( P \restriction A \) is continuous for some \( A \subseteq \pre{\omega}{(\omega+1)} \), then \( P(A) \) is nowhere dense in \( \pre{\omega}{\omega} \) (\cite[Lemma 5.4]{cmps}): since \( \pre{\omega}{\omega} \), being Polish, is Baire, this implies that \( P \) is not \(\omega\)-decomposable. Notice that such argument actually proves the stronger result that \( P \) cannot be decomposed into less than \( {\rm cov}(\mathcal{M}) \)-many continuous functions, where \( {\rm cov}(\mathcal{M}) \) is the smallest cardinality of a family of meager sets covering \( \pre{\omega}{\omega} \). On the other hand, the argument based on Theorem~\ref{theor:luzin0}\ref{theor:luzin0-1} has the merit of further showing (when combined in the obvious way with Lemma~\ref{lemma3})
that \emph{every} countable power of a 
\( \boldsymbol{\Sigma}^0_2 \)-measurable discontinuous function is not \(\omega\)-decomposable: using this and Solecki's dichotomy theorem~\ref{theorsoleckidichotomy} one gets as a byproduct that all such functions cannot be decomposed into less than \( {\rm cov}(\mathcal{M}) \)-many continuous functions.

Theorem~\ref{theor:luzin0}\ref{theor:luzin0-1} and Corollary~\ref{corPn} also show that the generalized Pawlikowski functions \( P_n \) (see Definition~\ref{defPn}) have properties analogous to \( P \).

\begin{corollary} \label{corPnindecomposable}
Let \( 1 \leq n < \omega \). Then \( P_n \) is a Baire class \( n \) (equivalently, \( \boldsymbol{\Sigma}^0_{n+1} \)-measurable) function which is not \(\omega\)-decomposable in \( \Sigma_n \) functions.
\end{corollary}

In contrast, we do not know if  results similar to Theorem~\ref{theorsoleckidichotomy} 
hold for the generalized Pawlikowski functions as well.

\begin{question}
Let \( X,Y \) be separable metrizable spaces with \( X \) analytic, and fix  \( 1 < n < \omega \). Is it true that for every \( \boldsymbol{\Sigma}^0_{n+1} \)-measurable function \( f \colon X \to Y \), either \( f \) is \(\omega\)-decomposable into \( \Sigma_n \) functions, or else%
\footnote{By Proposition~\ref{propPn}, the validity of this statement is independent of the choice of the set \( C_n \) used to define \( P_n \) (see Definition~\ref{defPn}).}
 \( P_n \sqsubseteq f \)?
\end{question}

Consider now functions from a Polish space \( X \) into a separable metrizable space \( Y \).
By Lemma~\ref{lemma3} and Theorem~\ref{theor:luzin0}\ref{theor:luzin0-2},  by considering the smallest \(n \geq 1 \) such that \( f \in \Sigma_n(X,Y) \) one easily obtains that
\emph{every} countable power of a  
discontinuous finite level Borel function \( f \colon X \to Y \) is not \(\omega\)-decomposable and, similarly,  
\emph{every} countable power of a non \( \boldsymbol{\Sigma}^0_m \)-measurable finite level Borel function \( f \colon X \to Y \) is not 
\(\omega\)-decomposable in \( \Sigma_m \) functions. When \( X \) is uncountable, 
this observation and Proposition~\ref{proppower} imply that for finite level Borel countable powers \( f_\omega \colon \pre{\omega}{X} \to \pre{\omega}{Y} \), being \(\omega\)-decomposable in \( \Sigma_m \) functions already implies that the whole \( f_\omega  \) is \( \boldsymbol{\Sigma}^0_m \)-measurable, and hence that the following curious result holds.

\begin{corollary} 
Let \( X \) be an uncountable Polish space and \( Y \) be a separable metrizable space. Then for every \( f \colon X \to Y \) and \( 1 \leq n < \omega \), if \( f \) (or, equivalently, \( f_\omega \)) is a finite level Borel function, then the following are equivalent:
\begin{enumerate-(a)}
\item
\( f_\omega \) is \(\omega\)-decomposable in \( \boldsymbol{\Sigma}^0_n \)-measurable functions;
\item
\( f_\omega \) itself is \( \boldsymbol{\Sigma}^0_n \)-measurable;
\item
\( f \) is \( \boldsymbol{\Sigma}^0_n \)-measurable.
\end{enumerate-(a)}
\end{corollary}

Finally, using Theorem~\ref{theor:luzin0}\ref{theor:luzin0-1} we can also provide further simple counterexamples (between arbitrary uncountable analytic spaces \( X,Y \) and for every \( \alpha \geq 1 \)) to  the generalized Lusin's question~\ref{questionLusin}.

\begin{corollary}\label{cor:luzin} 
Let \( X,Y \) be uncountable analytic spaces, and let \( \alpha > 1 \). Then \( \Sigma_\alpha(X,Y) \setminus \Dec(<\alpha)(X,Y) \neq \emptyset \), that is: there is a function \( f \in \Sigma_\alpha(X,Y) \) which is not \(\omega\)-decomposable in \( \Sigma_{\beta} \) functions for any \( 1 \leq \beta < \alpha \).

Moreover, when \( \alpha \) is additively closed such an \( f \) may be taken in \( \D_\alpha(X,Y) \), so that in this case we further get  \( \D_\alpha(X,Y) \setminus \Dec(<\alpha)(X,Y) \neq \emptyset \)
\end{corollary}

\begin{proof}
Let us first consider the case where \( \alpha = \delta+1 \) for some \( \delta \geq 1 \). By Corollary~\ref{cor3} there is some \( f  \in \Sigma_{\alpha}(X,Y) \) which is a proper \( \Sigma_{\alpha +\gamma, 1+\gamma} \) function for every \( \gamma < \omega_1 \). Setting \( \gamma = 2 \) and applying Theorem~\ref{theor:luzin0}\ref{theor:luzin0-1} with \( \alpha \) replaced by \( \delta \), we get that \( f \) is not \(\omega\)-decomposable in \( \Sigma_\delta \) functions and hence \( f \notin \Dec(<\alpha)(X,Y) \).

Assume now that \( \alpha \) is limit, let \( \seq{\alpha_n}{n \in \omega} \) be an increasing sequence of successor ordinals \( < \alpha \) cofinal in \( \alpha \), and let \( \seq{X_n}{n \in \omega} \) be a partition of \( X \) into \emph{uncountable} \( \boldsymbol{\Delta}^0_2 (X) \)-pieces. For each \( n \in \omega \), use the first part to find \( f_n \in \Sigma_{\alpha_n}(X_n,Y) \setminus \Dec(< \alpha_n)(X_n,Y) \): then \( f = \bigcup_{n \in \omega} f_n \) is as required.

Finally, observe that when \( \alpha \) is additively closed, any function \( f \) constructed as in the previous paragraph is in \( \D_\alpha(X,Y) \) by Fact~\ref{fct:basicSigmaalphabeta} and \( \Sigma_{\alpha_n}(X_n,Y) \subseteq \D_\alpha(X_n,Y) \). 
\end{proof}

Let us conclude this section with a small observation concerning extension of the Jayne-Rogers theorem~\ref{theorjaynerogers} to infinite levels.
As noticed in~\cite[p.~45]{motamenable}, no 
``reasonable'' generalization of the Jayne-Rogers theorem can hold at level \( \omega 
\) --- for example one cannot guarantee that all \( \boldsymbol{\Delta}^0_\omega \)-functions can be decomposed into countably many \( \boldsymbol{\Delta}^0_n \)-functions (for \( 1 \leq n < \omega \)). However, one could still speculate that e.g.\ 
every \( \boldsymbol{\Delta}^0_{\omega+1} \)-function is \(\omega\)-decomposable into \( 
\boldsymbol{\Delta}^0_\omega \)-functions with \( \boldsymbol{\Delta}^0_{\omega+1} \)-domain (i.e.\ that \( \D_{\omega+1} = \D^\omega_{\omega+1} \)), and similar conjectures can in principle be proposed for every \( \alpha \geq 1 \).
Nevertheless, the last part of Corollary~\ref{cor:luzin} shows that the set of levels  at which such decomposition theorems can hold is actually very sparse: there is a closed unbounded \( C 
\subseteq \omega_1 \) (namely the club of all additively closed countable ordinals) such 
that for every \( \alpha \in C \), the class \( \D_\alpha(X,Y) \) contains functions which are essentially indecomposable into countably many strictly simpler functions.

\section{Generalizations of the Jayne-Rogers theorem to higher levels} \label{sectionJR}

We begin with a simple lemma which computes the complexity of the graph of 
certain \(\omega\)-decomposable Borel functions; to prove it, we will use the following well-known result.

\begin{fact}[Folklore] \label{factPi}
Let \( X,Y \) be topological spaces with \( Y \) Hausdorff. Then \( \gr(f) \in \boldsymbol{\Pi}^0_\alpha(X \times Y ) \) for every \( f  \in \Sigma_\alpha(X,Y) \).
\end{fact}

\begin{lemma} \label{lemmagraph}
Let \( X,Y \) be topological spaces with \( Y \) Hausdorff, and let \( f \in \D^\beta_\alpha(X,Y) \) for \( \alpha,\beta \geq 1 \).
\begin{enumerate-(a)}
\item \label{lemmagraph-1}
If \( \beta \geq \alpha \) then \( \gr(f) \in \boldsymbol{\Pi}^0_\beta(X \times Y) \).
\item \label{lemmagraph-2}
If \( \beta < \alpha \) then \( \gr(f) \in \boldsymbol{\Delta}^0_\alpha(X \times Y) \).
\end{enumerate-(a)}
\end{lemma}

\begin{proof}
If \( \beta \geq \alpha \), then \( \D^\beta_\alpha(X,Y) = \Sigma_\beta(X,Y) \) by Lemma~\ref{lemmadefdecomposable}, hence \( \gr(f) \in \boldsymbol{\Pi}^0_\beta(X \times Y) \) by Fact~\ref{factPi}. Assume now \( \beta < \alpha \). Let \( \seq{ X_n }{ n \in \omega } \) witness \( f \in \D^\beta_\alpha(X,Y) \). Then by Fact~\ref{factPi}, \( \beta < \alpha \), and the choice of the \( X_n \)'s we get
\[
 \gr(f \restriction X_n) \in \boldsymbol{\Pi}^0_\beta(X_n  \times Y) \subseteq \boldsymbol{\Delta}^0_\alpha(X_n \times Y) \subseteq \boldsymbol{\Delta}^0_\alpha(X \times Y).
\]
Therefore, 
\[ 
\gr(f) = \bigcup\nolimits_{n \in \omega} \gr(f \restriction X_n) \in \boldsymbol{\Sigma}^0_\alpha(X \times Y ) 
\] 
and 
\[ 
(X \times Y) \setminus \gr(f) = \bigcup\nolimits_{n \in \omega} ((X_n \times Y) \setminus \gr(f \restriction X_n)) \in \boldsymbol{\Sigma}^0_\alpha(X,Y). \qedhere
 \] 
\end{proof}

Using Lemma~\ref{lemmagraph} and Proposition~\ref{propdefinabledec}\ref{propdefinabledec-1}, we easily get
the following characterization the class \( \D^1_\alpha(X,Y) \) for \( \alpha \geq 3 \).

\begin{theorem}\label{theormain} 
Let \( X , Y \) be metrizable spaces, and \( f \) be a function between \( X \) and \( Y \). For \( \alpha \geq 3 \), the following are equivalent:
\begin{enumerate-(a)}
\item \label{theormain-1}
\( \gr(f) \in \boldsymbol{\Sigma}^0_\alpha (X \times \widetilde{Y}) \) and \( f \) is \(\omega\)-decomposable;
\item \label{theormain-1'}
\( \gr(f) \in \boldsymbol{\Delta}^0_\alpha (X \times \widetilde{Y}) \) and \( f \) is \(\omega\)-decomposable;
\item \label{theormain-2}
\( f \in \D^1_\alpha(X,Y)  \).
\end{enumerate-(a)}
In other words,
\begin{equation} \label{eqmain}
\tag{$\dagger$}
\begin{split}
\D^1_\alpha(X,Y) &= \{ f \colon X \to Y \mid \gr(f) \in \boldsymbol{\Sigma}^0_\alpha(X \times \widetilde{Y}) \}^\Dec \\
& = \{ f \colon X \to Y \mid \gr(f) \in \boldsymbol{\Delta}^0_\alpha(X \times \widetilde{Y}) \}^\Dec.
\end{split}
\end{equation}
\end{theorem}

\begin{proof}
\ref{theormain-1}\( \Rightarrow \)\ref{theormain-2} by Proposition~\ref{propdefinabledec}\ref{propdefinabledec-1}.
To show \ref{theormain-2}\( \Rightarrow \)\ref{theormain-1'} simply note that 
since \( f \in \D^1_\alpha(X, Y) \subseteq \D^1_\alpha(X, \widetilde{Y}) \), then \( \gr(f) \in \boldsymbol{\Delta}^0_\alpha(X \times \widetilde{Y}) \) by Lemma~\ref{lemmagraph}\ref{lemmagraph-2}. The implication \ref{theormain-1'} \( \Rightarrow \) \ref{theormain-1} is obvious, hence we are done. 
\end{proof}

\begin{remark} \label{remmain-1}
Theorem~\ref{theormain} can be seen as a generalization to higher levels of~\cite[Theorem 10]{jayrog}, which states that if \( X,Y \) are metrizable spaces with \( Y \) \emph{\(\sigma\)-compact}, then \( f \in \D^1_2(X,Y) \) if and only if \( \gr(f) \in \boldsymbol{\Sigma}^0_2(X \times Y) \) (this last condition is equivalent to \( \gr(f) \in \boldsymbol{\Sigma}^0_2(X \times \widetilde{Y}) \) because \( Y \), being \(\sigma\)-compact, is in \( \boldsymbol{\Sigma}^0_2(\widetilde{Y}) \)). However, note that Theorem~\ref{theormain}, contrarily to~\cite[Theorem 10]{jayrog}, requires the extra hypothesis that \( f \) be \(\omega\)-decomposable; in contrast, we did not require any extra hypothesis on \( X \) and \( Y \) apart from metrizability (in particular, they can also be taken to be nonseparable).
\end{remark}

\begin{remark} \label{remmain-2} 
Since~\ref{theormain}\ref{theormain-2} only depends on the topological spaces \( X , Y \), Theorem~\ref{theormain} shows that also conditions~\ref{theormain-1}--\ref{theormain-1'} are independent of the particular choice of the completely metrizable space \( \widetilde{Y} \) in which \( Y \) is embedded: in fact, if \( \widetilde{Y}' \) were any other completely metrizable space with \( Y \subseteq \widetilde{Y}' \), then
 \[
 \gr(f) \in \boldsymbol{\Sigma}^0_\alpha(X \times \widetilde{Y})  \iff \gr(f) \in \boldsymbol{\Sigma}^0_\alpha(X \times \widetilde{Y}') \iff \gr(f) \in \boldsymbol{\Delta}^0_\alpha(X \times \widetilde{Y}')
\] 
for all \(\omega\)-decomposable functions \( f \colon X \to Y \).
Similar observations hold also for the subsequent statements involving spaces of the form \( \widetilde{Y} \) for some metrizable \( Y \).
\end{remark}

We will now restrict our attention to the special case of finite level Borel functions. Our first goal is to show that all functions in \( \D_{< \omega}(X,Y) = \bigcup_{n \in \omega} \D_n(X,Y) \) are \(\omega\)-decomposable (Lemma~\ref{lemmafinite}), and for this we first need to generalize
 Solecki's dichotomy theorem~\ref{theorsoleckidichotomy} to this wider class of 
functions. Recently, using an elegant (but rather involved) technique, Pawlikowski and Sabok have in fact extended
Theorem~\ref{theorsoleckidichotomy} to \emph{all} Borel 
functions from an analytic space \( X \) to a separable metrizable space \( Y \)~\cite[Theorem 1.1]
{pawsabrev}: however, when restricting ourselves to \emph{finite level} Borel functions, a much more elementary inductive proof can be given.

\begin{lemma}\label{lemmasoleckidichotomyfinite}
Let \( X,Y \) be separable metrizable spaces with \( X \) analytic. For every finite level Borel function \( f \colon X \to Y \), either \( f \) is \(\omega\)-decomposable, or else \( P \sqsubseteq f \).
\end{lemma}

\begin{proof}
We may assume without loss of generality that \( f \) is discontinuous, so
let \( 2 \leq n < \omega \) be such that \( f \in \Sigma_n(X,Y) \) and argue by induction on \( n \). The case \( n = 2 \) is 
the original Solecki's dichotomy theorem~\ref{theorsoleckidichotomy}. For the inductive step, assume that \( f 
\in \Sigma_{n+1}(X,Y) \), let \(\tau\) be the topology of \( X \), and let \( 
 \{ B_k \mid k \in \omega \} \) be a countable basis for the 
topology of \( Y \). Let \( \mathcal{A} = \{ A_{k,i} \mid k,i \in \omega \} 
\subseteq \boldsymbol{\Sigma}^0_n(X,\tau) \) be such that \( f^{-1}(B_k) = 
\bigcup_{i \in \omega} (X \setminus A_{k,i}) \) for every \( k \in \omega \): by Lemma~\ref{lemmachangetopology}, there is topology \( \tau \subseteq \tau' \subseteq 
\boldsymbol{\Sigma}^0_n(X,\tau) \) such that \( Z = (X,
\tau') \) is still analytic and \( \mathcal{A} \subseteq \tau' \). Therefore \( f \in 
\Sigma_2(Z,Y) \), and 
applying Theorem~\ref{theorsoleckidichotomy} we get that either \( P 
\sqsubseteq f \) (when \( f \) is considered as a function from \( Z \) to \( Y \)), 
or else \( f \in \Dec(Z,Y) \). In the former case, since \( \tau \subseteq \tau' \), \( \pre{\omega}{(\omega+1)} \) is compact, and \( X = (X,\tau) \) is Hausdorff, one easily gets that \( P \sqsubseteq f \) also when \( f \) is construed as a function from \( X \) to \( Y \). In the latter case we instead get that 
\( f \in \Sigma_2^\Dec(Z,Y) \), so that
\( f \in \D^1_3(Z,Y) \) by Corollary~\ref{cor+1}, and hence \( f \in \D^n_{n+2}
(X,Y) \) since \( \tau' \subseteq \boldsymbol{\Sigma}^0_n(X,\tau) \). Let \( 
\seq{ X_m }{ m \in \omega } \) witness this last fact: by the inductive hypothesis, 
for each \( m \in \omega \) either \( f \restriction X_m \in \Dec(X_m,Y) \), or 
else \( P \sqsubseteq f \restriction X_m \). If for all \( m \in \omega \) the first 
alternative holds, then \( f \) is \(\omega\)-decomposable as well;  if 
instead there is \( m \in \omega \) for which the second alternative holds, then 
\( P \sqsubseteq f \). Hence in both cases we are done.
\end{proof}

\begin{remark}
Using Corollary~\ref{cor+1}, Lemma~\ref{lemmasoleckidichotomyfinite}  can be sharpened as follows:
If \( f \colon X \to Y \) is a finite level Borel function and \( 1 \leq n < \omega \) is such that \( f \in \Sigma_n(X,Y) \), then either
\begin{enumerate-(I)}
\item \label{corsoleckidichotomy-1}
\( f \in \D^1_{n+1}(X,Y) \), or else
\item \label{corsoleckidichotomy-2}
\( P \sqsubseteq f \).
\end{enumerate-(I)}
(In particular, Solecki's dichotomy theorem~\ref{theorsoleckidichotomy} may be reformulated as: if \( f \in \Sigma_2(X,Y) \), then either \( f \in \D^1_3(X,Y) \) or else \( P \sqsubseteq f \).) 
Moreover, using Corollary~\ref{corPn}, Lemma~\ref{lemmaproperpreserved},  and Theorem~\ref{theor:luzin0}\ref{theor:luzin0-1} together with Corollary~\ref{corfinite} below, one can check that condition~\ref{corsoleckidichotomy-1} is equivalent to each of the following:
\begin{enumerate-(Ia)}
\item \label{corsoleckidichotomy-1a}
\( f \in \D_m(X,Y) \) for some \( m \in \omega \);
\item \label{corsoleckidichotomy-1b}
\( f \) is \(\omega\)-decomposable,
\end{enumerate-(Ia)}
while condition~\ref{corsoleckidichotomy-2} 
is equivalent to each of the following:
\begin{enumerate-(IIa)}
\item \label{corsoleckidichotomy-2b}
\( f \) is not a  \( \D_{n+1,n+1} \) function,
\item \label{corsoleckidichotomy-2c}
\( f \) is not \(\omega\)-decomposable.
\end{enumerate-(IIa)}
If \( n = 2 \),~\ref{corsoleckidichotomy-2} is also equivalent to
\begin{enumerate-(IIc)}
\item
\( f \) is a proper \( \Sigma_{2+\gamma,1+\gamma} \) function for every \( \gamma < \omega_1 \).
\end{enumerate-(IIc)}
Clearly, similar reformulations can be obtained also for the more general~\cite[Theorem 1.1]{pawsabrev}.
\end{remark}

\begin{lemma} \label{lemmafinite}
Let \( X , Y\) be separable metrizable spaces with \( X \) analytic. For every  \( f \in \D_{< \omega} (X,Y) \), \( f \) is \(\omega\)-decomposable. In particular, \( \D_n(X,Y) = \D^\Dec_n(X,Y) \) for every \( 1 \leq n < \omega \). 
\end{lemma}

\begin{proof}
Fix \( 1 \leq n < \omega \) and let \( f \in \D_n(X,Y) \subseteq \Sigma_n(X,Y)\). If \( f \) were not \(\omega\)-decomposable, then \( P \sqsubseteq f \) 
by Lemma~\ref{lemmasoleckidichotomyfinite}: but this would imply \( P \in 
\D_n(\pre{\omega}{(\omega+1)},\pre{\omega}{\omega}) \), contradicting Corollary~\ref{corPn}.  
\end{proof}

Using Corollary~\ref{cor+1} and Lemma~\ref{lemmafinite} we obtain:

\begin{corollary} \label{corfinite}
Let \( X , Y\) be separable metrizable spaces with \( X \) analytic.
\begin{enumerate-(a)}
\item
\( \D_n \subseteq \D^1_{n+1} \) for all \( 1 \leq n < \omega \).
\item
If \( f \colon X \to Y \) is a finite level Borel function, then \( f \) is \(\omega\)-decomposable if and only if \( f \in \D_n(X,Y) \) for some \( n \in \omega \).
\end{enumerate-(a)}
\end{corollary}

The Jayne-Rogers theorem~\ref{theorjaynerogers} can be seen as a characterization of the class \( \D^1_2(X,Y) \) (for appropriate \( X,Y \)): in this respect, the following result may be considered as a (partial) generalization to all finite levels of that theorem.

\begin{theorem}\label{theorJR}
Let \( X , Y\) be separable metrizable spaces with \( X \) analytic. Let \( 1 < n < \omega \), and \( f \colon X \to Y \) be such that \( \gr(f) \in \boldsymbol{\Sigma}^0_n(X \times \widetilde{Y}) \). Then the following are equivalent:
\begin{enumerate-(a)}
\item \label{theorJR-1}
\( f \in \D_n(X,Y) \);
\item \label{theorJR-2}
\( f \in \D^1_n(X,Y) \).
\end{enumerate-(a)}
In fact, 
\begin{equation} \label{eqrestrictedJR}
\tag{$\ddagger$}
\begin{split}
\D^1_n(X,Y) & = \D_n(X,Y) \cap \{ f \colon X \to Y \mid \gr(f) \in \boldsymbol{\Sigma}^0_n(X \times \widetilde{Y}) \} \\
& = \D_n(X,Y) \cap \{ f \colon X \to Y \mid \gr(f) \in \boldsymbol{\Delta}^0_n(X \times \widetilde{Y}) \}.
\end{split}
 \end{equation}
\end{theorem}

\begin{proof}
For the first part, the case \( n = 2 \) follows from the Jayne-Rogers theorem~\ref{theorjaynerogers}, hence we can assume \( n \geq 3 \). Then \ref{theorJR-1}\( \Rightarrow \)\ref{theorJR-2} by the assumption that \( \gr(f) \in \boldsymbol{\Sigma}^0_n(X \times \widetilde{Y}) \), Lemma~\ref{lemmafinite}, and Theorem~\ref{theormain}, while \ref{theorJR-2}\( \Rightarrow \)\ref{theorJR-1} by Lemma~\ref{lemmadefdecomposable}.

For the second part, \( \D_n(X,Y) \cap \{ f \colon X \to Y \mid \gr(f) \in \boldsymbol{\Delta}^0_n(X \times \widetilde{Y}) \} \subseteq \D_n(X,Y) \cap \{ f \colon X \to Y \mid \gr(f) \in \boldsymbol{\Sigma}^0_n(X \times \widetilde{Y}) \} \) is obvious, and \( \D_n(X,Y) \cap \{ f \colon X \to Y \mid \gr(f) \in \boldsymbol{\Sigma}^0_n(X \times \widetilde{Y}) \} \subseteq \D^1_n(X,Y) \) by the first part. Finally, \( \D^1_n(X,Y) \subseteq \D_n(X,Y) \) by Lemma~\ref{lemmadefdecomposable}, while \( \D^1_n(X,Y) \subseteq \D^1_n(X, \widetilde{Y}) \subseteq \{ f \colon X \to Y \mid \gr(f) \in \boldsymbol{\Delta}^0_n(X \times \widetilde{Y}) \} \) by Lemma~\ref{lemmagraph}\ref{lemmagraph-2}.
\end{proof}

By Theorem~\ref{theorJR}, the weak generalization of the Jayne-Rogers 
theorem (Conjecture~\ref{conjweakJR}) for the level \( 1 < n < \omega \)  is 
true when restricted to functions having 
\( \boldsymbol{\Sigma}^0_n \)-graph and hence, in particular, to \( 
\boldsymbol{\Sigma}^0_{n-1} \)-measurable functions (or, even more 
generally, to functions in \( \D^m_n(X,Y) \) for some \( 1 \leq m < n \)).

\begin{corollary} \label{corsimplecases} 
Let \( X , Y\) be separable metrizable spaces with \( X \) analytic, \( 1 \leq m < n < \omega \), and \( f \in \D^m_n(X,Y) \). Then the following are equivalent:
\begin{enumerate-(a)}
\item
 \( f \in \D_n(X,Y) \);
\item
\(  f \in \D^1_n(X,Y) \). 
\end{enumerate-(a)}
In particular, the conclusion holds for functions in \( \Sigma_{n-1}(X,Y) \).
\end{corollary}

\begin{proof}
By Lemma~\ref{lemmagraph}, every \( f \in 
\D^m_n(X,Y) \subseteq \D^m_n(X, \widetilde{Y}) \) has graph in \( 
\boldsymbol{\Sigma}^0_n(X \times \widetilde{Y}) \), hence the result follows 
from Theorem~\ref{theorJR}. The second part follows from the obvious  fact 
that \( \Sigma_{n-1}(X,Y) \subseteq \D^{n-1}_n(X,Y) \). 
\end{proof}

Corollary~\ref{corsimplecases} allows to derive 
Theorem~\ref{theorsemmes}\ref{theorsemmes-2} from 
Theorem~\ref{theorsemmes}\ref{theorsemmes-1} in a considerably shorter 
and more elementary way 
with respect to the original Semmes' argument contained in~\cite[Chapter 5]{semmesthesis}.  In 
fact, \( \D^1_3(\pre{\omega}{\omega},\pre{\omega}{\omega}) \subseteq 
\D_3(\pre{\omega}{\omega},\pre{\omega}{\omega}) \) by 
Lemma~\ref{lemmadefdecomposable}. Conversely, \( \D_3(\pre{\omega}
{\omega},\pre{\omega}{\omega}) \subseteq \Sigma_{3,2}(\pre{\omega}
{\omega},\pre{\omega}{\omega}) = \D^2_3(\pre{\omega}{\omega},
\pre{\omega}{\omega}) \) by Theorem~\ref{theorsemmes}\ref{theorsemmes-1}, whence \( \D_3(\pre{\omega}{\omega},
\pre{\omega}{\omega}) \subseteq \D^1_3(\pre{\omega}{\omega},
\pre{\omega}{\omega}) \) by Corollary~\ref{corsimplecases}.

More generally, Corollary~\ref{corsimplecases} implies that 
Conjecture~\ref{conjweakJR} for  the level \( 3 \leq m < \omega \) would 
automatically follow  by \emph{any} ``decomposition theorem'' for a class of 
the form \( \Sigma_{m,n} \) (\( 1 < n < m \)) appearing in 
Conjecture~\ref{conjstrongJR} --- see Question~\ref{quest1} and the 
discussion following it.

\medskip

Lemma~\ref{lemmagraph} and Theorem~\ref{theorJR} together show that the weak generalization of the Jayne-Rogers theorem is equivalent to the following conjecture:

\begin{conjecture} \label{conjweak}
Let \( X , Y\) be separable metrizable spaces with \( X \) analytic. Then for every 
\( 1 < n < \omega \), all functions in \( \D_n(X,Y) \) have 
\( \boldsymbol{\Sigma}^0_n(X \times \widetilde{Y}) \) graph.
\end{conjecture}

Notice that since Lemma~\ref{lemmagraph} implies
 \( \D^1_n(X,Y) \subseteq \D^1_n(X, \widetilde{Y}) 
\subseteq \{ f \colon X \to \widetilde{Y} \mid \gr(f) \in 
\boldsymbol{\Sigma}^0_n(X \times \widetilde{Y}) \} \), 
Conjecture~\ref{conjweak} is true for the case 
\( n=2 \) by the Jayne-Rogers theorem~\ref{theorjaynerogers}, and also 
for \( n = 3 \) in case \( 
X,Y\) are Polish spaces of topological dimension \( \neq \infty \)  by (the generalization of) 
Semmes' theorem~\ref{theorsemmes}\ref{theorsemmes-2}. 

These observations suggest the following general questions:

\begin{question} \label{quest1}
Let \( X , Y\) be separable metrizable spaces with \( X \) analytic. Let \( \alpha >1 \) and \( f \in \Sigma_\alpha(X,Y) \).
\begin{enumerate-(a)}
\item \label{quest1-1}
Does \( f \in \D_\alpha(X,Y) \) imply \( \gr(f) \in \boldsymbol{\Sigma}^0_\alpha(X \times \widetilde{Y}) \)?
\item \label{quest1-2}
More generally, for  which \( 1 < \beta \leq \alpha \)  does \( f \in \Sigma_{\alpha,\beta}(X,Y) \) imply \( \gr(f) \in \boldsymbol{\Sigma}^0_\alpha(X \times \widetilde{Y}) \)? In particular, does \( f \in \Sigma_{\alpha,2}(X,Y) \) already imply \( \gr(f) \in \boldsymbol{\Sigma}^0_\alpha(X \times \widetilde{Y}) \)?
\item \label{quest1-3}
Is \( \gr(f) \) a \emph{proper \( \boldsymbol{\Pi}^0_\alpha(X \times \widetilde{Y}) \) set} whenever \( f \) is a \emph{proper \( \Sigma_\alpha \) function}?
\end{enumerate-(a)}
\end{question}

Obviously, a positive answer to any of Question~\ref{quest1}\ref{quest1-1} or Question~\ref{quest1}\ref{quest1-2} 
for finite \( \alpha \)'s would imply Conjecture~\ref{conjweak}, and hence also Conjecture~\ref{conjweakJR}. 
Moreover, the second part of Question~\ref{quest1}\ref{quest1-2} is equivalent to the converse of 
Question~\ref{quest1}\ref{quest1-3}, namely to the following statement:
if \( \gr(f) \) is a proper \( \boldsymbol{\Pi}^0_\alpha(X \times \widetilde{Y}) \) set then \( f  \) is a proper \( 
\Sigma_\alpha \) function. 
To see this, notice that \( f \) is \emph{not} a proper \( \Sigma_\alpha \) function if and only if it is \( 
\boldsymbol{\Delta}^0_\alpha(X) \)-measurable (i.e.\ \( f^{-1}(U) \in \boldsymbol{\Delta}^0_\alpha(X) \) for every 
open \( U \subseteq Y \)), if and only if \( f \in \Sigma_{\alpha,2} \).
Concerning Question~\ref{quest1}\ref{quest1-3}, a sufficient condition for \(f \in \Sigma_\alpha(X,Y) \) having a 
proper \( \boldsymbol{\Pi}^0_\alpha(X \times \widetilde{Y} ) \) graph is that of requiring \( f^{-1}(y) \) to be a 
proper \( \boldsymbol{\Pi}^0_\alpha(X) \) set for some \( y \in Y \). However such condition is not 
necessary, as if e.g.\ \( X = Y = \pre{\omega}{\omega} \), then for every \( \alpha \geq 1 \) there are injective \( f \in \Sigma_\alpha(X,Y) \) with  proper 
\( \boldsymbol{\Pi}^0_\alpha(X \times \widetilde{Y}) \) graph.

As discussed after Conjecture~\ref{conjweak},  the Jayne-Rogers theorem~\ref{theorjaynerogers} and the (generalization of) 
Semmes' theorem~\ref{theorsemmes} (together with Lemma~\ref{lemmagraph}) answer positively Question~\ref{quest1}\ref{quest1-1}--\ref{quest1-2} in the case \( \alpha \leq 3 \) (for many interesting \( X,Y \)):
below we provide also a positive answer to Question~\ref{quest1}\ref{quest1-3} for \( \alpha = 2 \) for the special case  of a \(\sigma\)-compact \( Y \). 

\begin{proposition} \label{propnew} 
Let \( X \) be an analytic metrizable space and \(Y \) be a \(\sigma\)-compact metrizable space. For every \( f \in \Sigma_2(X,Y) \), the following are equivalent:
\begin{enumerate-(a)}
\item \label{propnew-1}
\( f \) is a proper \( \Sigma_2 \) function;
\item \label{propnew-2}
\( \gr(f) \) is a proper \( \boldsymbol{\Pi}^0_2(X \times \widetilde{Y}) \) set;
\item \label{propnew-3}
\( f \notin \D^1_2(X,Y) \);
\item \label{propnew-4}
\( f \notin \D_2(X,Y) \).
\end{enumerate-(a)}
\end{proposition}

\begin{proof}
\ref{propnew-1}\( \iff \)\ref{propnew-4} is obvious,
while \ref{propnew-4}\( \iff \)\ref{propnew-3} by the Jayne-Rogers theorem~\ref{theorjaynerogers}. 
Moreover, \ref{propnew-2}\( \Rightarrow \)\ref{propnew-3}, as if 
\( f \in \D^1_2(X,Y) \subseteq \D^1_2(X, \widetilde{Y}) \), then 
\( \gr(f) \in \boldsymbol{\Sigma}^0_2(X \times \widetilde{Y}) \) by 
Lemma~\ref{lemmagraph} (which can be applied since \( Y \), being metrizable and \( \sigma \)-compact, is also separable). Finally, \ref{propnew-3}\( \Rightarrow \)\ref{propnew-2} by~\cite[Theorem 10]{jayrog}.
\end{proof}

The proof of Proposition~\ref{propnew} heavily relies on~\cite[Theorem 10]{jayrog}, 
which is specific for functions in \( \D^1_2(X,Y) \) and apparently does not admit 
straightforward full generalizations to higher levels (see Remark~\ref{remmain-1}): nevertheless, is we restrict our attention to \(\omega\)-decomposable functions we can replace~\cite[Theorem 10]{jayrog} with Theorem~\ref{theormain} to get a partial answer to Question~\ref{quest1}\ref{quest1-3} also for 
\( \alpha = 3 \), as shown in the next proposition. Notice also that such proposition further generalizes Semmes' theorem~\ref{theorsemmes}\ref{theorsemmes-1} (for the special case of functions in \( \Dec(X,Y) \)) by 
adding to it two more equivalent conditions, namely the negations of~\ref{prop:semmes-1} and~\ref{prop:semmes-2}.

\begin{proposition} \label{prop:semmes}
Let \( X \) be a Polish space of topological dimension \( \neq \infty \), and 
\( Y \) be a separable metrizable space. If \( f \in \Sigma_3(X,Y) \) and \( f \) is 
\(\omega\)-decomposable, then the following are equivalent:
\begin{enumerate-(a)}
\item \label{prop:semmes-1}
\( f \) is a proper \( \Sigma_3 \) function;
\item \label{prop:semmes-2}
\( \gr(f) \) is a proper \( \boldsymbol{\Pi}^0_3(X \times \widetilde{Y}) \) set;
\item \label{prop:semmes-3}
\( f \notin \D^{2}_3(X,Y) \);
\item \label{prop:semmes-4}
\( f \notin \Sigma_{3,2}(X,Y) \).
\end{enumerate-(a)}
\end{proposition}

\begin{proof}
\ref{prop:semmes-1}\( \iff \)\ref{prop:semmes-4} is obvious (see the paragraph after Question~\ref{quest1}), while
\ref{prop:semmes-4}\( \iff \)\ref{prop:semmes-3} by (the discussion following) Semmes'
theorem~\ref{theorsemmes}\ref{theorsemmes-1}. \ref{prop:semmes-2}\( \Rightarrow \)\ref{prop:semmes-3}
because if \( f \in \D^2_3(X,Y) \subseteq \D^2_3(X, \widetilde{Y}) \) then \( \gr(f) \in \boldsymbol{\Sigma}^0_3(X \times \widetilde{Y}) \) by 
Lemma~\ref{lemmagraph}, so it remains only to show that \ref{prop:semmes-3}\( \Rightarrow \)\ref{prop:semmes-2}. Suppose 
towards a contradiction that \( \gr(f) \in \boldsymbol{\Sigma}^0_3(X \times \widetilde{Y}) \): since \( f \) is assumed to be 
\(\omega\)-decomposable, then \( f \in \D^1_3(X,Y) \subseteq \D^2_3(X,Y) \) by Theorem~\ref{theormain}, contradicting~\ref{prop:semmes-3}.
\end{proof}

The next proposition shows that the kind of functions considered in Proposition~\ref{prop:semmes} always exists as long as \( Y \) is not discrete (Proposition~\ref{propproperanddecomposable} fails if \( Y \) is  discrete because in this case if \( f \in \Sigma_\alpha(X,Y) \) then \( f^{-1}(A) \in \boldsymbol{\Delta}^0_\alpha(X) \) \emph{for every} \( A \subseteq Y \), hence there are no proper \( \Sigma_\alpha \) functions at all).

\begin{proposition} \label{propproperanddecomposable}
Let \( X  \) be an uncountable analytic space and \( Y \) be
an infinite nondiscrete metrizable space. Then for every \( \alpha \geq 1 \)
there is a proper \( \Sigma_\alpha \) function
\( f \colon X \to Y \) which is \(\omega\)-decomposable.
\end{proposition}

\begin{proof}
Let \( \seq{ y_n }{ n \in \omega } \) be a sequence of points in \( Y \) converging to some \( y \in Y \) and let \( A 
\subseteq X \) be a proper \( \boldsymbol{\Sigma}^0_\alpha(X) \) set. Let \( A_n \in 
\boldsymbol{\Delta}^0_\alpha(X) \)
be such that \( n \leq m \Rightarrow A_n \subseteq A_m \) and \( A = \bigcup_{n \in \omega} 
A_n \), and set \( A_{-1} = \emptyset \). Define \( f \colon X \to Y \) by setting \( f(x) = y_n \) if \( x \in A \) and \( n \in \omega \) is smallest such that \( x \in 
A_n \), and \( f(x) = y \) otherwise. 

We first show that \( f \in \Sigma_\alpha(X,Y) \). Let \( U \subseteq Y \) be open: setting  \( I  = \{ i \in  \omega \mid y_i \in U \} \), it is straightforward to check that if \( y 
\notin U \) then \( f^{-1}(U) = \bigcup_{i \in I} (A_i \setminus A_{i-1}) \), while if \( y \in U \) then 
\( f^{-1}(U) = X \setminus \bigcup_{i \notin I} (A_i \setminus A_{i-1}) \). 
Since \( \omega \setminus I \) is finite whenever \( y \in U \), 
in both cases we get \( f^{-1}(U) \in \boldsymbol{\Sigma}^0_\alpha(X) \). 

To see that \( f \) is in fact a proper \( \Sigma_\alpha \) 
function, notice that \( f^{-1}(Y \setminus \{ y \}) 
= A \). Finally, \( f \) is trivially \(\omega\)-decomposable because it has a countable range.
\end{proof}

This also suggests the following general question.

\begin{question} \label{questiondecomposable}
Let \( X,Y \) be uncountable analytic spaces. Given \( 1 \leq n \leq m < \omega \), let \( \widehat{\Sigma}_{m,n}(X,Y) \) be the class of functions from \( \Sigma_{m,n} (X,Y) \) which do not appear in any other finite level Borel class which does not already contain \( \Sigma_{m,n}(X,Y) \), i.e.
\[ 
\widehat{\Sigma}_{m,n}(X,Y) = \Sigma_{m,n}(X,Y) \setminus \bigcup	\{ \Sigma_{m',n'}(X,Y) \mid m' < m \vee m'-m' < m-n \}.
 \] 
(By Theorem~\ref{theorinclusion}\ref{theorinclusion-1} each \( \widehat{\Sigma}_{m,n}(X,Y) \) is nonempty.)
Are there \(\omega\)-decomposable \( f \in \widehat{\Sigma}_{m,n}(X,Y) \)?
\end{question}

Proposition~\ref{propproperanddecomposable} and Lemma~\ref{lemmafinite} give a positive answer to Question~\ref{questiondecomposable} if either \( n = 1 \) or \( m = n \), but the general case remains open.

\section{A reformulation of the strong generalization of the Jayne-Rogers theorem} \label{secstrongJR}

In this section we will show that the following statement is equivalent to the strong generalization of the Jayne-Rogers theorem (Conjecture~\ref{conjstrongJR}).

\begin{conjecture}\label{conjstrong} 
For every analytic space \( X \), every separable metrizable space \( Y \), every \( 1 < n < m < \omega \), and every \( f \in \Sigma_{m,n}(X,Y) \), the topology \( \tau \) of \( X \) can be refined by a topology \( \tau' \subseteq \boldsymbol{\Sigma}^0_2(X,\tau) \) such that \( Z = (X, \tau') \) is an analytic space and \( f \in \Sigma_{m-1,n}(Z,Y) \).
\end{conjecture}

Notice that Conjecture~\ref{conjstrong} makes sense only for finite level Borel classes: for every \( 1 \leq n,k  < \omega \) and every \( f \in \Sigma_{\omega,n}(X,Y) \) there is no refinement \( \tau' \) of the topology \( \tau \) of \( X \) with the property that \( \tau' \subseteq \boldsymbol{\Sigma}^0_k(X,\tau) \), \( Z = (X,\tau') \) is analytic, and \( f \in \Sigma_{< \omega, n}(Z,Y) \).

\begin{lemma}[Folklore] \label{lemmalowering}
Let \( (X,\tau) \) be a topological space, \( 1 < m < \omega \), and \( S \in \boldsymbol{\Sigma}^0_m(X, \tau) \). 
Then there is a countable family \( \mathcal{P} =  \{ P_i \mid i \in \omega \} \subseteq \boldsymbol{\Pi}^0_1(X, \tau) \) such that \( 
S \in \boldsymbol{\Sigma}^0_{m-1}(X,\tau') \) for every topology \( \tau' \) on \( X \) such that \( \mathcal{P} \subseteq \tau' \).
\end{lemma}

\begin{proof}
By induction on \( m \geq 2 \).
\end{proof}

\begin{lemma}\label{lemmaonedirection}	
Let \( X, Y \) be separable metrizable spaces with \( X \) analytic. For every \( 1 < n < m < \omega \) and \( f \in \D^{m-n+1}_m(X,Y) \) the topology \(\tau\) of \( X \) can be refined by a topology \( \tau' \subseteq \boldsymbol{\Sigma}^0_2(X, \tau	) \) such that \( Z = (X ,\tau') \) is an analytic space and \( f \in \D^{m-n}_{m-1}(Z,Y) \).
\end{lemma}

\begin{proof}
Let 
\( \seq{ X_k }{ k \in \omega } \) witness \( f \in \D^{m-n+1}_m(X,Y) \). Let 
\( \{ B_l \mid l \in \omega \} \) be a countable basis for \( Y \), and for 
\( k , l \in \omega  \) let \( S_{k,l} \in \boldsymbol{\Sigma}^0_{m-n+1}(X,\tau) \) 
be such that \( (f \restriction X_k)^{-1}(B_l) = f^{-1}(B_l) \cap X_k =  S_{k,l} \cap X_k  \). Using 
Lemma~\ref{lemmalowering}, pick for every \( k,l \in \omega \) families 
\( \mathcal{P} = \{ P_{k,l,i} \mid i \in \omega \} \) and 
\( \widetilde{\mathcal{P}} = \{ \widetilde{P}_{k,i} \mid i \in \omega \} \) of (\(\tau\)-)closed sets such that 
\( S_{k,l} \in \boldsymbol{\Sigma}^0_{m-n}(X,\tau') \) and 
\( X_k \in \boldsymbol{\Sigma}^0_{m-1}(X,\tau') \) (hence 
\( X_k \in \boldsymbol{\Delta}^0_{m-1}(X,\tau') \), since  \( \seq{ X_k }{ k \in \omega } \) is 
a countable partition of \( X \)) for every topology \( \tau ' \supseteq \mathcal{P} \cup \widetilde{\mathcal{P}} \). By Lemma~\ref{lemmachangetopology}, there is a refinement \( \tau' \) of \( \tau \) such that 
\(  \mathcal{P} \cup \widetilde{\mathcal{P}}  \subseteq \tau' \subseteq \boldsymbol{\Sigma}^0_2(X,\tau) \) and \( Z = (X,\tau') \) is  an analytic space: therefore  
\( S_{k,l} \in \boldsymbol{\Sigma}^0_{m-n}(Z) \) and 
\( X_k \in \boldsymbol{\Delta}^0_{m-1}(Z) \), whence 
\( f \in \D^{m-n}_{m-1}(Z,Y) \).
\end{proof}

\begin{theorem} \label{theorsgJR}
Let \( X, Y \) be separable metrizable spaces with \( X \) analytic. Then Conjecture~\ref{conjstrong} is equivalent to the strong generalization of the Jayne-Rogers theorem (i.e.\ to Conjecture~\ref{conjstrongJR}).
\end{theorem}

\begin{proof}
Assume first that Conjecture \ref{conjstrong} is true: we want to show that \( \Sigma_{m,n}(X,Y)  = \D^{m-n+1}_m(X,Y) \) for every \( 1 < n \leq m < \omega \).
Lemma~\ref{lemmadefdecomposable} implies \( \D^{m-n+1}_m(X,Y) \subseteq \Sigma_{m,n}(X,Y) \) for every \( 1 < n \leq m < \omega \), hence only the reverse inclusions 
\begin{equation}\label{eqreverseinclusion}
 \Sigma_{m,n}(X,Y) \subseteq \D^{m-n+1}_m(X,Y) 
\end{equation}
need to be considered. 
We argue by induction 
on \( 2 \leq m < \omega \).
If \( m=2 \) then \( m = n \) and the corresponding instantiation of~\eqref{eqreverseinclusion} is \( 
\Sigma_{2,2}(X,Y) = \D_2(X,Y) \subseteq \D^1_2(X,Y) \), which is true by the Jayne-Rogers theorem~\ref{theorjaynerogers}.
So let us assume that \( m 
\geq 3 \) and that, by inductive hypothesis, for every \( X',Y' \) separable metrizable spaces with \( X' \) analytic and every \(1 < n \leq m-1 \) it holds that
 \begin{equation} \label{eqinductive}
\Sigma_{m-1,n}(X',Y')  \subseteq \D_{m-1}^{m-n}(X',Y').
\end{equation}
Let us first assume that \( n<m\). Given \( f \in \Sigma_{m,n}
(X,Y) \), by Conjecture \ref{conjstrong} and~\eqref{eqinductive}  
we get \( 
f \in \D_{m-1}^{m-n}(Z,Y) \) (where \( Z = (X,\tau') \) is as in the statement of Conjecture~\ref{conjstrong}), i.e.\ 
that there is  a countable 
partition \( \seq{ X_k }{ k \in \omega } \) of \( X \) in \( 
\boldsymbol{\Delta}^0_{m-1}(X,\tau') \)-pieces such that \( f \restriction X_k \in 
\Sigma_{m-n}(Z_k,Y) \) for every \( k \in \omega \), where \( Z_k \) denotes the space \( X_k \) endowed with the relativization of \( \tau' \). Since \( \tau' \subseteq \boldsymbol{\Sigma}^0_2(X,\tau) \),  \( \seq{  X_k }{ k \in \omega } \) witnesses \( f \in \D^{m-n+1}_m(X,Y) \) as desired. Now we consider the unique remaining case, i.e.\ we show that \( \Sigma_{m,m}(X,Y) \subseteq \D^1_m(X,Y) \). Since we already showed that~\eqref{eqreverseinclusion} holds for all \( n<m \),  setting \( n = m-1 \) we get that 
\[  
\Sigma_{m,m}(X,Y) \subseteq \Sigma_{m,m-1}(X,Y)  \subseteq \D^2_m(X,Y) . 
\]
Since \( \Sigma_{m,m}(X,Y) = \D_m(X,Y) \) and \( m \geq 3 \) by case assumption, this implies \( \Sigma_{m,m}(X,Y) \subseteq \D_m^1(X,Y) \) by Corollary~\ref{corsimplecases}, as required.

Conversely, assume that Conjecture~\ref{conjstrongJR} 
holds. Let \( 1 < n < m < \omega \) and let \( \tau \) be the topology of \( X \). 
Given \( f \in \Sigma_{m,n}(X,Y) = \D^{m-n+1}_m(X,Y) \), applying Lemma~\ref{lemmaonedirection} we get a topology \( \tau \subseteq \tau' \subseteq \boldsymbol{\Sigma}^0_2(X,\tau) \) such that \( Z = (X,\tau') \) is an analytic space and \( f \in \D^{m-n}_{m-1}(Z,Y) \), so that  \( f \in \Sigma_{m-1,n}(Z,Y) \) by 
Lemma~\ref{lemmadefdecomposable}. Therefore Conjecture~\ref{conjstrong} holds as well.
\end{proof}

It follows from the proof of Theorem~\ref{theorsgJR} that every instantiation of Conjecture~\ref{conjstrongJR} for 
parameters \( 1 < n < m < \omega \) implies the corresponding instantiation of Conjecture~\ref{conjstrong} (for the 
same parameters).
Using this fact, we get that Conjecture~\ref{conjstrong} is true for \( m \leq 3 \) (i.e.\ for \( n = 2 \) and \( m = 3 \)) 
when \( X \) is a Polish space of dimension \( \neq \infty \) and \( Y\) is a zero-dimensional metrizable space by (the 
generalization of) Semmes' theorem~\ref{theorsemmes}\ref{theorsemmes-1}.

Finally, as a corollary to the proof of Theorem~\ref{theorsgJR}, one can show that Conjecture~\ref{conjstrong} for the special case \( n = m-1 \) already suffices to prove the weak generalization of the Jayne-Rogers theorem (Conjecture~\ref{conjweakJR}).

\begin{corollary}
Suppose that for every analytic space \( X' \), every separable metrizable space \( Y' \), every \( 2  < m < \omega \), and every \( f \in \Sigma_{m,m-1}(X',Y') \), the topology \( \tau \) of \( X' \) can be refined by a topology \( \tau' \subseteq \boldsymbol{\Sigma}^0_2(X',\tau) \) such that \( Z = (X', \tau') \) is an analytic space and \( f \in \Sigma_{m-1,m-1}(Z,Y') \). Then \( \D_m(X,Y) = \D^1_m(X,Y) \) for every \( 1 < m < \omega \) and every separable metrizable spaces \( X,Y \) with \( X \) analytic.
\end{corollary}

\begin{proof}
We argue by induction on \( m \geq 2 \). The basic case \( m=2 \) is the 
Jayne-Rogers theorem~\ref{theorjaynerogers}. Assume now that \( 
\D_m(X,Y) = \D^1_m(X,Y) \): we want to show that \( \D_{m+1}(X,Y) = 
\D^1_{m+1}(X,Y) \). If \( f \in \Sigma_{m+1,m}(X,Y) \), by our assumption 
we have \( f \in \Sigma_{m,m}(Z,Y) \), where \( Z = (X, \tau')\) is as in the 
hypotheses of the corollary. Then \( f \in \D^1_m(Z,Y) \) by inductive 
hypothesis, whence \( f \in \D^2_{m+1}(X,Y) \). This shows \( 
\Sigma_{m+1,m}(X,Y) \subseteq \D^2_{m+1}(X,Y) \) (in fact, \( 
\Sigma_{m+1,m}(X,Y) = \D^2_{m+1}(X,Y) \) by 
Lemma~\ref{lemmadefdecomposable}). Therefore, since \( \D_{m+1}(X,Y) \subseteq 
\Sigma_{m+1,m}(X,Y) \) we get \( \D_{m+1}(X,Y) = \D^1_{m+1}(X,Y) \) 
by Corollary~\ref{corsimplecases}.
\end{proof}

The strong generalization of the Jayne-Rogers theorem (equivalently, 
Conjecture~\ref{conjstrong}) would have a remarkable influence on the 
structure of finite level Borel classes. For example, by 
Lemma~\ref{lemmagraph} it would provide a positive answer to both 
Question~\ref{quest1}\ref{quest1-1} and Question~\ref{quest1}\ref{quest1-2} for every \( 1 < \alpha < \omega \). Moreover it would imply 
that given arbitrary \( 1 < n < m < \omega \), \emph{every} proper 
\( \Sigma_{m,n} \) function \( f \colon X \to Y \) (where \( X,Y \) are separable metrizable 
spaces with \( X \) analytic) cannot be \(\omega\)-decomposable. In fact,  by 
Lemma~\ref{lemmagraph} we would get 
\( \gr(f) \in \boldsymbol{\Sigma}^0_m(X \times \widetilde{Y}) \)  for every 
\( f \in \Sigma_{m,n}(X,Y) \subseteq \Sigma_{m,n}(X, \widetilde{Y}) = \D^n_m(X,\widetilde{Y}) \) (where the last equality is obtained by applying the strong generalization of the Jayne-Rogers theorem). Therefore if such an \( f \) is also \(\omega\)-decomposable, then \( f \in \D^1_m(X,Y) \subseteq \D_m(X,Y) \) by 
Theorem~\ref{theormain} (since \( m \geq 3 \)), whence \( f \) cannot be a 
proper \( \Sigma_{m,n} \) function. In other words, for \( m \geq 3 \) we 
would get that for every \( f \in \Sigma_m^\Dec(X,Y) \) 
\begin{equation} \tag{$*$} \label{eqdichotomy}
\text{either } f \in \D^1_m(X,Y) \text{ or } f \text{ is a proper } \Sigma_m \text{ function}. 
\end{equation}
In particular, since all functions \( f \in \widehat{\Sigma}_{m,n}(X,Y) \) are necessarily proper \( \Sigma_{m,n} \) functions, this would completely answer Question~\ref{questiondecomposable}: \( \widehat{\Sigma}_{m,n}(X,Y) \) contains \(\omega\)-decomposable functions if and only if \( n = 1 \) or \( m =n \). 

\begin{remark}
If \( m = 2 \), the unrestricted version of dichotomy~\eqref{eqdichotomy}  (i.e.\ when considering all functions in \( \Sigma_2(X,Y) \), without restricting to \(\omega\)-decomposable functions) is 
true by the Jayne-Rogers theorem~\ref{theorjaynerogers}, and in fact it is easily seen to be equivalent to it. 
In 
contrast, notice that for \( m \geq 3 \) the restriction to \(\omega\)-decomposable 
functions when looking for~\eqref{eqdichotomy} cannot be 
relaxed: any function in \( \widehat{\Sigma}_{m,m-1}(X,Y) \) (which is nonempty by Theorem~\ref{theorinclusion}\ref{theorinclusion-1}) is in \( \Sigma_m(X,Y) \), but it is neither in \( \D^1_m(X,Y) \) nor a proper \( \Sigma_m \) function.
\end{remark}

\section{Characterizing finite level Borel functions as compositions of simpler functions} \label{seccomposition}

The next theorem shows that when \( 1 < n < m < \omega \) the functions in \( \D^n_m(X,Y) \) can be decomposed into simpler functions 
also in a different sense, namely in terms of composition of two (strictly) simpler functions. 
It turns out that such decomposition actually characterizes the mentioned class.

\begin{theorem} \label{theorcomposition}
Let \( X,Y \) be separable metrizable spaces, and \( 1 \leq n \leq m < \omega \). For every \( f \colon X \to Y \) the following are equivalent:
\begin{enumerate-(a)}
\item \label{theorcomposition-1}
\( f \in \D^n_m(X,Y) \);
\item \label{theorcomposition-2}
\( f = h \circ g \) with \( g \in \Sigma_n(X,Z) \) and \( h \in \D^1_{m-n+1}(Z,Y) \) for some separable metrizable space \( Z \).
\end{enumerate-(a)}
Moreover, the space \( Z \) in~\ref{theorcomposition-2} may be chosen to be zero-dimensional, and if \( X \) is analytic (respectively, Polish) then \( Z \) can be taken to be analytic (respectively, Polish) as well.
\end{theorem}

\begin{proof}
We first show \ref{theorcomposition-1}\( \Rightarrow \)\ref{theorcomposition-2} arguing as in the proof of 
Theorem~\ref{theorsgJR}. Let \( \tau \) be the topology of \( X \), and let \( \seq{ X_k }{ k \in \omega } \) witness  \( f \in \D^n_m(X,Y) \). Let \( \{ B_l \mid l \in \omega \} \) be a countable basis for \( Y \), and let \( 
S_{k,l} \in \boldsymbol{\Sigma}^0_n(X,\tau) \) be such that \( (f \restriction X_k)^{-1}(B_l) = f^{-1}(B_l) \cap X_k 
= S_{k,l} \cap X_k  \) for every \( k,l \in \omega \).
 By repeatedly applying \( n-1 \)-times Lemma~\ref{lemmalowering}, we get that there are countable collections \( \mathcal{P} = \{ 
P_{k,l,i} \mid i \in \omega \}  \)  and \( \widetilde{\mathcal{P}} =  \{ \widetilde{P}_{k,i}  \mid i \in \omega \} \) of \( 
\boldsymbol{\Pi}^0_{n-1}(X) \)-sets such that \( X_k \in \boldsymbol{\Sigma}^0_{m-n+1}(X, \tau') \) and \( S_{k,l} 
\in \boldsymbol{\Sigma}^0_1(X,\tau') \) (for all \( k,l \in \omega \)) for every topology \( \tau' \) on \( X \) with \( \mathcal{P} \cup \widetilde{\mathcal{P}} \subseteq \tau' \). By Lemma~\ref{lemmachangetopology}, we can find such a \( \tau' \) with the 
further properties that \( Z = (X,\tau') \) is (zero-dimensional) separable metrizable and \( \tau \subseteq \tau' \subseteq 
\boldsymbol{\Sigma}^0_n(X,\tau) \). Moreover, if \( X \) is analytic (respectively, Polish), then \( \tau' \) can be chosen so that \( Z = (X,\tau') \) is analytic (respectively, Polish) as well. It follows that \( f \in \D^1_{m-n+1}(Z,Y) \) and that the identity function \( 
\id \) on \( X \) belongs to \( \Sigma_n(X,Z) \). Therefore, setting \( g = \id \colon X \to Z \) and \( h = f \colon 
Z \to Y \) we get~\ref{theorcomposition-2}.

Conversely, assume that~\ref{theorcomposition-2} holds, and let \( \seq{ Z_k }{ k \in \omega } \) witness \( h \in \D^1_{m-n+1}(Z,Y) \). Then \( X_k = g^{-1}(Z_k) \in \boldsymbol{\Delta}^0_m(X) \) and \( f \restriction X_k = (h \restriction Z_k) \circ (g \restriction X_k) \in \Sigma_n(X_k,Y) \). Therefore \( f \in \D^n_m(X,Y) \), i.e.\ \ref{theorcomposition-1} holds.
\end{proof} 

Using (the generalization of) Semmes' theorem~\ref{theorsemmes}\ref{theorsemmes-1} on \( \Sigma_{3,2} \) functions, setting \( m= 3 \) and \( n = 2 \) in Theorem~\ref{theorcomposition} we get the following corollary.

\begin{corollary} \label{corcomposition}
Let \( X \) be a Polish space of topological dimension \( \neq \infty \), and let \( Y \) be a zero-dimensional separable metrizable space. For every \( f \colon X \to Y \) the following are equivalent:
\begin{enumerate-(a)}
\item \label{corcomposition-1}
\( f \in \Sigma_{3,2}(X,Y) \);
\item \label{corcomposition-2}
\( f = h \circ g \) with \( g \in \Sigma_2(X,Z) \) and \( h \in \D_2(Z,Y) \) for some zero-dimensional Polish space \( Z \).
\end{enumerate-(a)}
\end{corollary}

When \( X = Y = \pre{\omega}{\omega} \), Corollary~\ref{corcomposition} assumes the particularly elegant form
\[ 
\Sigma_{3,2} = \D_2 \circ \Sigma_2,
 \] 
where for simplicity of notation we put \( \Sigma_{3,2} = \Sigma_{3,2}
(\pre{\omega}{\omega},\pre{\omega}{\omega}) \), \( \D_2 = 
\D_2(\pre{\omega}{\omega},\pre{\omega}{\omega}) \), \( \Sigma_2 = 
\Sigma_2(\pre{\omega}{\omega},\pre{\omega}{\omega}) \), and  \( \D_2 
\circ \Sigma_2 = \{ h \circ g \mid {h \in \D_2} \wedge {g \in \Sigma_2} \} \).
This is because every zero-dimensional Polish space \( Z \) is homeomorphic 
to a closed subset of \( \pre{\omega}{\omega} \) 
by~\cite[Theorem 7.8]{kec}, and every \( \boldsymbol{\Delta}^0_2 \)-function from a \( 
\boldsymbol{\Delta}^0_2(\pre{\omega}{\omega}) \) (hence, in particular, a closed) 
subset of \( \pre{\omega}{\omega} \) to \( \pre{\omega}{\omega} \) can be 
trivially extended to a function in \( \D_2(\pre{\omega}{\omega},\pre{\omega}
{\omega}) \).

Theorem~\ref{theorcomposition} actually shows that a further consequence of the strong generalization of the Jayne-Rogers theorem (Conjecture~\ref{conjstrongJR}) would be the full characterization of all finite level Borel classes in terms of composition of functions belonging to Borel classes of a lower level, namely:
\begin{quote}
Let \( X \) be an analytic space and \( Y \) be separable metrizable. For every \( 1 < n < m < \omega \) and \( f \colon X \to Y \) the following are equivalent:
\begin{enumerate-(a)}
\item
\( f \in \Sigma_{m,n}(X,Y) \);
\item
\( f = h \circ g \) with \( g \in \Sigma_{m-n+1}(X,Z) \) and \( h \in \D_n(Z,Y) \) for some analytic space \( Z \).
\end{enumerate-(a)}
\end{quote}
In other words: all the ``nonstandard'' finite level Borel classes \( \Sigma_{m,n}(X,Y) \) (\( 1 < n < m < \omega \)) would be generated via composition by the more familiar Borel classes of the form \( \Sigma_k(X,Y) \) or \( \D_k(X,Y) \).

\section{An application to Banach spaces} \label{sectionBanach}

Let \( X \) be a Polish space and \( \alpha < \omega_1 \). Following~\cite{jayne}, let \( \mathcal{B}_\alpha(X) \) be the set of all  Baire class 
\(\alpha\) (equivalently, \( \boldsymbol{\Sigma}^0_{\alpha+1} \)-measurable) real-valued 
functions (in particular, \( \mathcal{B}_0(X) = C(X) \) is 
the class of all real-valued continuous functions on \( X \)), and denote by \( \mathcal{B}^*_\alpha(X) \) the collection of all 
bounded functions in \( \mathcal{B}_\alpha(X) \). When endowed in the obvious way with the pointwise  operations, the spaces \( 
\mathcal{B}_\alpha(X) \) and \( \mathcal{B}_\alpha^*(X) \) can be viewed as lattice-ordered algebras, rings, or multiplicative semigroups. 
Moreover,  each space \( \mathcal{B}^*_\alpha(X) \) is a  Banach space with respect to the supremum norm. These spaces are 
quite popular and well-studied in Banach space theory --- see e.g.~\cite{das,jayne,das2,shazaf, gonmar}: for example, 
in~\cite{das} it is shown that \( \mathcal{B}^*_\alpha(I) \) is not complemented in \( \mathcal{B}^*_\beta(I) \) for \( 0 < \alpha < \beta 
\) (whence \( \mathcal{B}^*_\alpha(I) \) is not linearly isometric to \( \mathcal{B}^*_\beta(I) \)), while  in~\cite{jayne} it is further shown 
that if \( X \) and \( Y \) are compact Polish 
spaces and at least one of them is uncountable, then \( \mathcal{B}^*_\alpha(X) \) is not linearly isometric to \( \mathcal{B}^*_\beta(Y) \) 
whenever \( \alpha \neq \beta \).  Consider now the following problem.

\begin{question} \label{questionBanach}
Given \( \alpha \geq 1 \) and two Polish spaces \( X \) and \( Y \), are the Banach spaces \( \mathcal{B}^*_\alpha(X) \) and \( \mathcal{B}^*_\alpha(Y) \) linearly isometric (in symbols, \( \mathcal{B}^*_\alpha(X) \cong_{\rm li} \mathcal{B}^*_\alpha(Y) \))?
\end{question}

A major progress in answering Question~\ref{questionBanach} is given by~\cite[Theorem 2]{jayne}. First of all, it is shown that \( \mathcal{B}^*_\alpha(X) \cong_{\rm li} \mathcal{B}^*_\alpha(Y) \) if and only if one of the following equivalent conditions hold:
\begin{enumerate-(a)}
\item
\( \mathcal{B}^*_\alpha(X) \) and \( \mathcal{B}^*_\alpha(Y) \) are isometric;
\item
\( \mathcal{B}^*_\alpha(X) \) and \( \mathcal{B}^*_\alpha(Y) \) are isomorphic as rings (equivalently, lattices, or multiplicative semigroups);
\item
\( \mathcal{B}_\alpha(X) \) and \( \mathcal{B}_\alpha(Y) \) are isomorphic as rings (equivalently, lattices, or multiplicative semigroups).
\end{enumerate-(a)}

A further equivalent condition from~\cite[Theorem 2]{jayne} is related to the following notion.

\begin{definition}	
Let \( X,Y \) be two topological spaces and \( \alpha \geq 1 \). We say that \( X \) and \( Y \) are \emph{\( \D_\alpha \)-isomorphic} (in symbols, \( X \simeq_\alpha Y \)) if there is a bijection \( f \colon X \to Y \) such that \( f \in \D_\alpha(X,Y) \) and \( f^{-1} \in \D_\alpha(Y,X) \).
\end{definition}

\begin{theorem}[Jayne] \label{theorjayne}
Let \( X,Y \) be Polish spaces.%
\footnote{Actually~\cite[Theorem 2]{jayne} holds in the broader context of realcompact spaces. Recall that every Polish space, being second-countable, is Lindel\"of, and hence also realcompact. A further extension of this result to perfectly normal spaces can be found in~\cite[Theorem 2.1]{shazaf}.}
For every \( \alpha \geq 1 \), \( \mathcal{B}^*_\alpha(X) \cong_{\rm li} \mathcal{B}^*_\alpha(Y) \) if and only if \( X \simeq_{\alpha+1} Y \).
\end{theorem}

Let \( 1 \leq \alpha \leq \beta \): since \( \D_{\alpha+1}(X,Y) \subseteq \D_{\beta + 1 }(X,Y) \) then \( X \simeq_{\alpha + 1} Y \Rightarrow X \simeq_{\beta + 1} Y \), and thus we get that if \( \mathcal{B}^*_\alpha(X) \cong_{\rm li} \mathcal{B}^*_\alpha(Y) \) then \( \mathcal{B}^*_\beta(X) \cong_{\rm li} \mathcal{B}^*_\beta(Y) \).
Theorem~\ref{theorjayne} then reduces Question~\ref{questionBanach} to the problem of analyzing the minimal possible 
complexity (in terms of the classes \( \D_\alpha \)) of an isomorphism between \( X \) and \( Y \), a problem which is also interesting \emph{per se}
--- see~\cite[Section 4]{motschsel} for more results on this topic. 

Assume first that \( 
X \) and \( Y \) are countable Polish spaces. If they have different cardinalities, 
then there is no bijection between them, whence \( 
\mathcal{B}^*_\alpha(X) \not\cong_{\rm li} \mathcal{B}^*_\alpha(Y) \) for every \( \alpha < \omega_1 \). If \( X , Y \) are finite and of 
the same cardinality, then they must be discrete and hence homeomorphic (whence \( \mathcal{B}^*_\alpha(X) 
\cong_{\rm li} \mathcal{B}^*_\alpha(Y) \) for every \( \alpha < \omega_1 \)). Finally, if both \( X, Y \) are infinite, then
 any bijection between \( X \) and \( Y \) witnesses \( X \simeq_2 Y \), so that in this case \( \mathcal{B}^*_\alpha(X) \cong_{\rm li} \mathcal{B}^*_\alpha(Y) \) for every \( \alpha \geq 1 
\). This bound cannot be further lowered in general because taking e.g.\ \( X = \omega \) with the discrete topology 
and \( Y = \omega +1 \) with the order topology we get \( X \not\simeq_1 Y \), and hence \(  \mathcal{B}^*_0(X) 
\not\cong_{\rm li} \mathcal{B}^*_0(Y) \).

Let us now consider the uncountable case. Since  between any two uncountable Polish spaces \(X , Y \) there is a Borel 
isomorphism such that \( f \in \Sigma_2(X,Y) \) and \( f^{-1} \in \Sigma_2(Y,X) \) (see~\cite[p.~212]{kur}), by \( 
\Sigma_2(X,Y) \subseteq \D_\omega(X,Y) \) and \( \Sigma_2(Y,X) \subseteq \D_\omega(Y,X) \) we get \( X 
\simeq_\omega Y \), and hence also \( X \simeq_{\omega+1} Y \). This implies \( \mathcal{B}^*_\alpha(X) \cong_{\rm li} 
\mathcal{B}^*_\alpha(Y) \) for every \( \alpha \geq \omega \), and therefore \(\omega\) provides a rough upper bound for the 
minimal \( \alpha < \omega_1 \) such that \( \mathcal{B}^*_\alpha(X) \cong_{\rm li} \mathcal{B}^*_\alpha(Y) \).
 In~\cite[Theorem 8.1]{jayrogisomorphismII} (see also~\cite[Theorem 4.21]{motschsel}), Jayne and Rogers essentially showed that if \( X , Y \) are further assumed to be both of topological 
dimension \( \neq \infty \), then this bound can be considerably lowered: in fact, in this case we get \( X \simeq_3 Y \), whence \( 
\mathcal{B}^*_\alpha(X) \cong_{\rm li} \mathcal{B}^*_\alpha(Y) \) for every \( \alpha \geq 2 \). This  bound cannot be further decreased: for example  \( \RR 
\not\simeq_2 \pre{\omega}{\omega} \) (since any \( f \in \D_2(\RR,\pre{\omega}{\omega}) = \D^1_2(\RR,\pre{\omega}{\omega}) \) maps \( K_\sigma \) sets to \( K_\sigma \) sets, and hence cannot be surjective),
thus \( \mathcal{B}^*_\alpha(\RR) \cong_{\rm li} \mathcal{B}^*_\alpha(\pre{\omega}{\omega}) \) if and only if  \( \alpha \geq 2 \). Similarly, it is not hard to check that \( \RR \simeq_2 I \) but \( \RR \not\simeq_1 I \), 
whence \( \mathcal{B}^*_\alpha(\RR) \cong_{\rm li} \mathcal{B}^*_\alpha(I) \) for \( \alpha \geq 1 \) but \( \mathcal{B}^*_0(\RR)  \not\cong_{\rm li} 
\mathcal{B}^*_0(I) \). 

Thus~\cite{jayrogisomorphismI, jayrogisomorphismII} completely answer Question~\ref{questionBanach} for the 
case of uncountable Polish spaces of topological dimension \( \neq \infty \), but the problem of determining whether \( 
\mathcal{B}^*_\alpha(X) \cong_{\rm li} \mathcal{B}^*_\alpha(Y) \) (for \(  \alpha < \omega \)) when one of \( X,Y \) has topological 
dimension \( \infty \) has remained open (at least to the author's best knowledge) until now. One may be tempted to 
conjecture that also in this case \( \mathcal{B}^*_\alpha(X) \cong_{\rm li} \mathcal{B}^*_\alpha(Y) \) for \( \alpha = 3 \), or at least for some 
finite \( \alpha \) (possibly depending on \( X \) and \( Y \)). The next result, which is an immediate and easy consequence of Lemma~\ref{lemmafinite}, refutes this 
conjecture and shows that the  upper bound of \( \omega \) obtained at the beginning of the previous paragraph cannot be lowered for arbitrary Polish spaces (even when restricting to compact spaces).

\begin{theorem}
Let \( X,Y \) be uncountable Polish spaces such that \( X \) has topological dimension \( \infty \) and \( Y \) has 
topological dimension \( \neq \infty \) (e.g.\ we can take \( X = [0,1]^\omega \) and \( Y = I \)). Then \( \mathcal{B}^*_n(X) \not\cong_{\rm li} \mathcal{B}^*_n(Y) \) for every \( n < \omega \).
\end{theorem}

\begin{proof}
By Theorem~\ref{theorjayne}, it is enough to prove that \( X \not\simeq_n Y \) for every \( 1 \leq n < \omega \). 
Since \( Y 
\cong_3 \pre{\omega}{\omega} \) by our assumption on \( Y \) and the mentioned~\cite[Theorem 8.1]{jayrogisomorphismII}, 
we 
can assume without loss of generality that \( Y = \pre{\omega}{\omega} \). By (the argument contained in) Example~\ref{exisomorphism}, there 
is no 
Borel isomorphism \( f \) between \( X \) and \( \pre{\omega}{\omega} \) such that both \( f \) and \( f^{-1} \) are 
\(\omega\)-decomposable. Since all functions in \( \D_n(X,\pre{\omega}{\omega}) \) and \( \D_n(\pre{\omega}
{\omega},X) \) 
are \(\omega\)-decomposable by Lemma~\ref{lemmafinite}, this implies \( X \not\simeq_n \pre{\omega}{\omega} \) 
for 
every \( 1 \leq n < \omega \), as required.
\end{proof}

It would be interesting to understand what happens if \( X \) and \( Y \) have both topological dimension \( \infty \). As 
observed in~\cite[Proposition 4.27]{motschsel}, if both \( X \) and \( Y \) are universal for Polish spaces, e.g.\ if \(X \) 
is the Hilbert cube \( [0,1]^\omega \) and \( Y = \RR^\omega \), then \( X \simeq_3 Y \), whence \( 
\mathcal{B}^*_\alpha(X) \cong_{\rm li} \mathcal{B}^*_\alpha(Y) \) for every \( \alpha \geq 2 \) (but notice that \( [0,1]^\omega \not\simeq_2 \RR^\omega \), hence \( \mathcal{B}^*_1(X) \not\cong_{\rm li} \mathcal{B}^*_1(Y) \)). In some cases, e.g.\ when \( X , Y \) are both compact universal Polish spaces, then one can also get \( X \simeq_2 Y \) (together with \( X \not\simeq_1 Y \)): an example is obtained by letting \( X  \) be \( [0,1]^\omega \) and \( Y \) be the disjoint union of two copies of \( [0,1]^\omega \). However, the general case remains open:

\begin{question}
Are there Polish spaces \( X, Y \), both of topological dimension \( \infty \), such that \( X \not\simeq_3 Y \)? Are there \( X, Y \) as before such that \( X \not\simeq_n Y \) for every \( 1 \leq n < \omega \)?
\end{question}

\end{document}